\documentclass[generic,preprint]{lms}

\usepackage[utf8]{inputenc}
\usepackage{xspace,enumerate}
\usepackage{graphicx}
\usepackage{subcaption}
\usepackage[pdftex]{hyperref}
\usepackage{amsfonts,amsmath}
\usepackage[linecolor=blue]{todonotes}

\newif\iflong
\longfalse
\longtrue

\iflong
\makeatletter
\gdef\@journal{%
  \vbox to 5.5\p@{\noindent
    \parbox[t]{\textwidth}{\raggedleft\normalfont\footnotesize\baselineskip 9\p@
     \ifblm {\itshape Bull.}\ \titlep@gedetails
       \else\ifjlm {\itshape J.}\ \titlep@gedetails
         \else\ifplm {\itshape Proc.}\ \titlep@gedetails
          \else {\itshape }  
        \fi
      \fi
    \fi%
  }
  \vss}%
}

\makeatother
\else
\fi

\DeclareMathOperator{\Rect}{Rect}

\newtheorem{theorem}{Theorem}[section]
\newtheorem{atheo}{Theorem}

\newtheorem{claim}[theorem]{Claim}
\newtheorem{lemma}[theorem]{Lemma}
\newtheorem{proposition}[theorem]{Proposition}
\newtheorem{corollary}[theorem]{Corollary}
\newtheorem{remark}[theorem]{Remark}
\newtheorem{definition}[theorem]{Definition}

\newcommand{\bbE}{{\ensuremath{\mathbb E}} }

\newcommand{\bbP}{{\ensuremath{\mathbb P}} }

\newcommand{\cA}{{\ensuremath{\mathcal A}} }
\newcommand{\cB}{{\ensuremath{\mathcal B}} }
\newcommand{\cC}{{\ensuremath{\mathcal C}} }
\newcommand{\cD}{{\ensuremath{\mathcal D}} }
\newcommand{\cE}{{\ensuremath{\mathcal E}} }
\newcommand{\cF}{{\ensuremath{\mathcal F}} }
\newcommand{\cG}{{\ensuremath{\mathcal G}} }
\newcommand{\cH}{{\ensuremath{\mathcal H}} }

\newcommand{\cK}{{\ensuremath{\mathcal K}} }

\newcommand{\cM}{{\ensuremath{\mathcal M}} }

\newcommand{\cO}{{\ensuremath{\mathcal O}} }

\newcommand{\cR}{{\ensuremath{\mathcal R}} }

\newcommand{\cU}{{\ensuremath{\mathcal U}} }

\newcommand{\cW}{{\ensuremath{\mathcal W}} }

\newcommand{\gep}{\varepsilon}       

\renewcommand{\tilde}{\widetilde}          

\newcommand{\R}{\mathbb{R}}

\newcommand{\Z}{\mathbb{Z}}
\newcommand{\N}{\mathbb{N}}

\def\bs{\boldsymbol}

\newcommand\bP{\ensuremath{\bs{\mathrm{P}}}}

\newcommand{\ind}{{\sf 1}}

\renewcommand{\epsilon}{\varepsilon}

\newenvironment{myitemize}{%
\begin{list}{$\bullet$}%
        {%
        \setlength{\itemsep}{0.4em}%
        \setlength{\topsep}{0.5em}%
        \setlength\leftmargin{2.45em}%
        \setlength\labelwidth{2.05em}%
        \setlength{\labelsep}{0.4em}%
        }%
        }%
{\end{list}}

\renewenvironment{itemize}{
\begin{myitemize}}%
{\end{myitemize}}

 \newcommand{\be}[1]{\begin{equation}\label{#1}}
 \newcommand{\ee}{\end{equation}}

 \newcommand{\bl}[1]{\begin{lemma}\label{#1}}
 \newcommand{\el}{\end{lemma}}

 \newcommand{\br}[1]{\begin{remark}\label{#1}}
 \newcommand{\er}{\end{remark}}

 \newcommand{\bt}[1]{\begin{theorem}\label{#1}}
 \newcommand{\et}{\end{theorem}}

 \newcommand{\bd}[1]{\begin{definition}\label{#1}}
 \newcommand{\ed}{\end{definition}}

 \newcommand{\bcl}[1]{\begin{claim}\label{#1}}
 \newcommand{\ecl}{\end{claim}}

 \newcommand{\bp}[1]{\begin{proposition}\label{#1}}
 \newcommand{\ep}{\end{proposition}}

 \newcommand{\bc}[1]{\begin{corollary}\label{#1}}
 \newcommand{\ec}{\end{corollary}}

 \newcommand{\bpr}{\begin{proof}}
 \newcommand{\epr}{\end{proof}}

 \newcommand{\bi}{\begin{itemize}}
 \newcommand{\ei}{\end{itemize}}

\newcommand{\esp}[1]{\mathbb{E}\etc{#1}}
\newcommand{\ens}[1]{\left\{#1\right\}}

\newcommand{\prob}[1]{\mathbb{P}_0\etp{#1}}
\newcommand{\probb}[1]{\mathbb{P}\etp{#1}}
\newcommand{\etc}[1]{\left [#1 \right ]}
\newcommand{\etp}[1]{\left (#1 \right )}
\newcommand{\valabs}[1]{\left|#1 \right|}
\newcommand{\norme}[1]{\left\Vert#1 \right\Vert}

\newcommand{\unsur}[1]{\frac{1}{#1}}
\newcommand{\floor}[1]{\lfloor#1\rfloor}

\newcommand{\probbeta}[1]{\bP_\beta\etp{#1}}

\newcommand{\esperance}[2]{\mathbb{E}_{#1}\etc{#2}}

\newcommand{\undemi}{\frac12}
\newcommand{\un}[1]{1_{\etp{#1}}}

\newcommand{\PP}{\mathbb{P}}

\newcommand{\probz}[1]{\mathbb{P}_{z}\etp{#1}}
\newcommand{\probx}[1]{\mathbb{P}_{x}\etp{#1}}

\newcommand{\snt}{\mathbf{S}_{\floor{tN}}}

\title[Limit RW excursion with typical area]{Limit theorems for Random Walk excursion  conditioned to {\color{black}enclose} a
  typical area}

\author{Philippe Carmona and Nicolas Pétrélis}

\classno{60G50 (primary), 82B41 (secondary).
}

\extraline{The authors thank the Centre Henri Lebesgue ANR-11-LABX-0020-01 for creating an attractive mathematical environment.}

\begin{document}
\maketitle

\begin{abstract}
We derive a functional central limit theorem for the excursion of a random walk  conditioned 
on enclosing a prescribed geometric area. We assume that the increments of the random walk are integer-valued, centered, with 
a third moment equal to zero and a finite fourth moment.  This result complements the work of \cite{DKW13} where local central limit theorems are provided for the geometric area
of the excursion of a symmetric random walk with finite second moments. 

Our result turns out to be a key tool to derive the scaling limit of the \emph{Interacting Partially-Directed Self-Avoiding Walk} at criticality 
which is the object of a companion paper \cite{CarPet17a}.  This requires to derive a reinforced version of our result in the case of 
a random walk with Laplace symmetric increments.

\end{abstract}

\section{Introduction and main results}
{\color{black}We will use the notation $\N_0=\N\cup\{0\}$}. We let also $(X_i)_{i\geq 1}$ be an IID sequence of integer-valued centered random variables of law $\bbP$ and of finite variance $\sigma^2$. We denote by $S:=(S_i)_{i=0}^\infty$ the random walk starting from the origin ($S_0=0$) of increments 
$(X_i)_{i=1}^{\infty}$, i.e., 
\be{defStiS} S_n := \sum_{i=1}^n X_i,\qquad \quad  (n\in \N),
\ee
and the algebraic area enclosed in the random walk by  
\be{defAA}
A_n=A_n(S) := \sum_{i=1}^n S_i\qquad \quad (n\in \N).
\ee
We consider positive excursions of the random walk and therefore we define the stopping time
$$ \tilde \tau := \inf\ens{n\ge 1: S_n \le 0}.$$
The scaling limit of the random walk excursion $(S_i)_{i=0}^{\tilde\tau}$ conditioned on $\tilde \tau=n$ as $n\to \infty$
is proven in \cite{CarCha13} and  \cite{So}. To be more specific, we denote by $(e_t, 0\le t\le 1)$  the
standard Brownian excursion \cite{Imhof84,PitYor96}
normalized
to have duration/length 1, that is $R_e=1$, where
\be{stanexc}
 R_e := \inf\ens{t>0 : e_t=0} =1.
\ee
{\color{black} For  every $x\in \N$,  we let $\bbP_x$ be the law of $(x+S_n)_{n\geq 0}$}. 
If $S$ is sampled from 
$\bbP_0(\cdot \, |\, \tilde \tau=n)$, then
$$\displaystyle\Big(
  \frac{S_{\floor{sn}}}{\sigma\sqrt{n}}; 0\leq s\leq 1 \Big)\quad 
\text{converges in distribution to}\quad 
\displaystyle\big(e_s; 0\leq s\leq 1\big),$$
in the space of cadlag functions endowed with the Skorohod topology.

In the present paper, we are interested in the scaling limit of positive excursions conditioned on their area rather than on their length. 
This was for a long time an open issue, mostly because standard Gnedenko techniques 
to derive local central limit theorems (LCLT) are difficult to apply when working with a random path 
constrained to remain positive. This difficulty was recently overstepped in \cite[Theorem 1.1]{DKW13} 
which states that 
\begin{equation}\label{DKW}
\sup_{a\in \N} \Big| n^{3/2}\, \bbP_0(A_n(S)=a\, |\, \tilde \tau=n+1)- \frac{1}{\sigma} w_{e}\Big(\frac{a}{n^{3/2}}\Big)\Big |=o(1),
\end{equation}
where $w_{e}$ is the density of the area enclosed by $e$ that is $\int_0^1 e_s ds$. 

In the present paper, we bring the analysis some steps further by displaying a functional central limit theorem for 
the whole trajectory of a random walk excursion conditioned on enclosing a prescribed area.  For technical reasons, we shall assume from now on that 
\be{defhyprw}
\bbE(X_1^4)<\infty \quad \text {and} \quad \bbE(X_1^3)=0.
\ee 
We believe that our four main Theorems remain true under the finite variance assumption {\color{black} but we are not able to prove it} \footnote{Our method of proof requires extra smoothness 
for the LCLT in (\ref{eq:smoothnessstandartllt}--\ref{eq:smoothnessstandartllt22})  and in Proposition \ref{lltalgarea}}.

Before stating our main Theorem, we need a few more notations. We define the pseudo inverse of the sequence of algebraic areas $(A_n(S))_{n\in \N}$ by 
\be{defchi}
 \chi_s := \inf\ens{n\ge 1: A_n(S) \ge s},\quad\quad  (s\geq 0).
 \ee
To identify the limiting process of the rescaled excursions we need to introduce
$\cE=(\cE_t,t\ge 0)$ the
Brownian excursion normalized by its area, defined in Corollary~\ref{cor:defnormedexcursionarea}. It is the analogue of the standard excursion $(e_t,0\le t\le 1)$ (recall \eqref{stanexc}) except that it is  normalized
to enclose a unit  area, i.e.,
\be{defEE}
\int_0^{R_\cE} \cE_s\, ds = 1 \quad \text{with} \quad R_{\cE}=\inf\ens{t>0 : \cE_t=0},
\ee
rather than to have a unit length. Finally, we set for $t\geq 0$ and $s\in [0,1]$
\be{defae} A_t(\cE)= \int_0^t \cE_s\, ds,\qquad a_\cE(s) = \inf\ens{t>0 : 
A_t(\cE)>s}.
\ee

\begin{atheo}
  \label{thm:a}
Let  $(X_i)_{i\geq 1}$ be an IID sequence of integer-valued centered random variables of variance $\sigma^2$ and satisfying \eqref{defhyprw}. As $L\to \infty$, with $S$ sampled under $\prob{\cdot \mid A_{\tilde \tau} = L, S_{\tilde \tau} =0}$, the
process
\begin{itemize}
\item 
$\displaystyle\Big(
  \frac{S_{\floor{s\tilde \tau}}}{\sigma\sqrt{\tilde \tau}}; 0\leq s\leq 1 \Big)$
converges in distribution to
$\displaystyle\Big( \tfrac{1}{\sqrt{R_{\cE}}}\, \cE_{sR_{\cE}}; 0\leq s\leq 1\Big)$.
\item $ \displaystyle\unsur{\sigma L^{1/3}} \big(S_{\chi_{sL}}; 0\le
  s \le 1\big)$
converges in distribution  to 
$ \displaystyle\big(\cE_{a_\cE(s)}; 0\le s\le 1\big)$.
\end{itemize}
\end{atheo}
Both convergences in Theorem \ref{thm:a} occur in the set of cadlag functions  endowed with the Skorohod topology.  In the first convergence the limiting process is a Brownian excursion normalized by its  area, also rescaled in time by its extension and in space by the root of this extension.  We actually prove a bit more in the paper since we establish
the convergence  of 
$(\frac{\tilde \tau}{L^{2/3}})_{L\in \N}$ towards $R_{\cE}$ jointly with the convergence of processes. This is no surprise since the typical area under an excursion of length $N$ should be
approximately $N^{3/2}$ therefore, the typical length of an excursion
of area $L$ should be approximately $L^{2/3}$ and the typical height, with a diffusive scaling, should be thus $L^{1/3}$. As a consequence, the following convergence also holds true and may 
appear more useful to the reader: 
$$\displaystyle\Big(
  \frac{S_{\floor{s L^{2/3}}\wedge \tilde \tau}}{\sigma L^{1/3}}; \  s\geq 0 \Big)
\quad \text{converges in distribution to} \quad 
\displaystyle\Big( \cE_{s\wedge R_{\cE}}; \  s\geq 0 \Big).$$

For the second convergence in Theorem \ref{thm:a}, we apply a non-linear time-change to the excursion, i.e., for every $s\in [0,1]$
we observe the excursion at the moment its area reaches $sL$. The resulting process (rescaled in space by $L^{1/3}$)
converges towards a Brownian excursion normalized by its area, subject to a similar time-change.

With Theorem \ref{thm:c} below, we characterize the limiting processes displayed in Theorem \ref{thm:a} and prove that their distribution are those of some Bessel bridge excursions, normalized by their length.
 To that purpose, we let  $Y$ be distributed as $(|B_{a_B(s)}|)_{0\le s \le 1}$
 with 
$$a_B(s) = \inf\Big\{t>0 : \int_0^t \valabs{B_u}\, du >s\Big\},$$
{\color{black} and we denote by $C_{1/3}$ the $1/3$-stable regenerative set (see e.g.  \cite[Appendix A]{CSZ} for the 
details of its construction).}
\begin{atheo}
  \label{thm:c}
  \begin{enumerate}
  \item $Y$ is distributed as $\big(\big(
  \frac{3}{2} \rho_{t}\big)^{2/3},t\in[0,1]\big)$ where 
 $(\rho_t)_{t\in [0,1]}$ is a Bessel process  of dimension
 $4/3$. In particular the zero set $Z(Y)=\{s\in [0,1]\colon\,
 Y(s)=0\}$ is distributed as  $C_{1/3}$. Furthermore, the law of $Y$ conditioned by $Y_1=0$ is
 well defined as the law of $\etp{\etp{\frac{3}{2}
     b_{t}}^{2/3},t\in[0,1]}$ with $b$ a Bessel bridge of dimension
 $4/3$, and the law of $Z(Y) \cap \etc{0,1}$ conditioned by $Y_1=0$ is
 the law of  $\tilde{C}_{1/3}=C_{1/3}\cap \etc{0,1}$
 conditioned on $1\in C_{1/3}$.

\item The process $(\cE_{a_\cE(s)},0\le s\le 1)$ has the law $\pi^Y$
  of the excursion of $Y$ normalized by its length.
  \end{enumerate}
\end{atheo}

We conclude this section with a brief outline of the rest of the paper.  With Section \ref{sec3}  we apply our main result in the framework of 
\footnote{Interacting Partially Directed Self Avoiding Walk}  IPDSAW a model introduced in \cite{ZL68} to study the collapse transition of a polymer surrounded by a poor solvent.
 We state in particular Theorem \ref{thm:ab}, which is an advanced version of Theorem \ref{thm:a} since it deals also with the random walk  $\bar S:=(\bar S_i)_{i=0}^\infty$
obtained by switching the sign of every other increment of $S$. Theorem \ref{thm:ab} turns out to be the key tool to derive the scaling limit 
of IPDSAW at
criticality, which is the object of a companion
paper \cite{CarPet17a}.
 In Section \ref{ps}, we provide a sketch for the proof of Theorem \ref{thm:a}, shedding some light on the 
links between our proof and some results from \cite{MR1340834} and from \cite{DKW13}.    Section \ref{sec:macdonald} is dedicated to the multi-dimensional generalization of a
non-standard technique initially displayed in \cite{MR1340834} to derive LCLT. Sections \ref{sec:preuvea}, \ref{excnorm} and \ref{excnorm2} are dedicated to the proof of Theorem \ref{thm:a}, \ref{thm:ab}, and  \ref{thm:c}, respectively.

\section{Extension and Application}\label{sec3}
Recall \eqref{defStiS} and define $\bar S:=(\bar S_i)_{i=0}^\infty$ by $\bar S_0=0$ and 
\be{defStiS2} 
\bar{S}_n := \sum_{i=1}^n
(-1)^{i+1} X_i\quad \text{for}\quad n\in \N.
\ee
We recall \eqref{defae} and that $\cE$ is a Brownian excursion normalized by its area. We let $B$ be a standard Brownian motion  
independent of $\cE$.
\begin{atheo}
  \label{thm:ab}
Let  $(X_i)_{i\geq 1}$ be an IID sequence of integer-valued centered random variables of variance $\sigma^2$ and satisfying \eqref{defhyprw}.  As $L\to \infty$, when sampled under $\prob{\cdot \mid A_{\tilde \tau} = L, S_{\tilde \tau} =0}$, the
process
\begin{itemize}
\item 
$\displaystyle\Big(
  \frac{S_{\floor{s\tilde \tau}}}{\sigma\sqrt{\tilde \tau}},  \frac{\bar{S}_{\floor{s\tilde \tau}}}{\sigma\sqrt{\tilde \tau}}; 0\leq s\leq 1 \Big)$
converges in distribution to
$\displaystyle\Big( \tfrac{1}{\sqrt{R_{\cE}}}\, \cE_{sR_{\cE}}, \tfrac{1}{\sqrt{R_{\cE}}}\,  B_{sR_{\cE}}; 0\leq s\leq 1\Big)$,
\item $ \displaystyle\unsur{\sigma L^{1/3}} \big(S_{\chi_{sL}},\bar{S}_{\chi_{sL}}; 0\le
  s \le 1\big)$
converges in distribution  to 
$ \displaystyle\big(\cE_{a_\cE(s)},B_{a_\cE(s)}; 0\le s\le 1\big)$.
\end{itemize}
\end{atheo}

 We refer to \cite{CNGPT18} for a recent topical review about IPDSAW, but let us introduce the model in a few words.
The IPDSAW is a polymer model whose configurations (of size $L\in \N$) are given by the 
$L$-step trajectories of a self-avoiding walk taking unit steps up, down, and right. Each such configurations is associated with an Hamiltonian given by $\beta>0$ (the coupling parameter) times the number {\color{black} of neighboring sites of the lattice visited non-consecutively by the walk. } The model is known to undergo a phase transition at some $\beta_c$ 
between an extended phase ($\beta<\beta_c$) and a collapsed phase $(\beta\geq\beta_c)$. In the extended phase, a typical configuration rescaled 
in time by $L$ and in space by $\sqrt{L}$  
converges in law towards a 1-dimensional Brownian motion. Inside the collapsed phase ($\beta>\beta_c$) the appropriate rescaling of a typical path is $\sqrt{L}$
in both time and space and the convergence  occurs in probability  towards a deterministic two-dimensional Wulff shape. 

At criticality  ($\beta=\beta_c$) deriving the scaling limit of 
a typical path is more involved since the limiting object turns out to be a  two-dimensional truly random subset of $\R^2$. 
This is the object of a companion paper \cite{CarPet17a} and the proof heavily relies on the fact that the critical-IPDSAW 
enjoys a random walk representation which allows us to reconstruct a typical trajectory  with the help of two random processes:  
\begin{enumerate}
\item  a profile whose law is that of a random walk conditioned on enclosing an area $\sim L$ 
\item a center-of-mass walk obtained by switching the sign of every other increments of the profile. 
\end{enumerate}
To be more specific, the random walk representation provides an IID sequence of random variables $(X_i)_{i\in \N}$ with a symmetric Laplace law $\bP_\beta$, i.e.,
\be{defincre}
\probbeta{X_1=k}= \unsur{c_\beta} e^{-\frac{\beta}{2} \valabs{k}},\ 
k\in \Z \quad\text{with}\quad c_\beta := \frac{1+ e^{-\frac\beta2}}{1- e^{-\frac\beta2}}.
\ee
We define $S$ and $\bar S$ with those increments (recall \eqref{defStiS} and \eqref{defStiS2}) and the geometric area enclosed by $S$ after $n$ steps 
is $ G_n := \sum_{i=1}^n \valabs{S_i}$ and its perturbation $K_n := n + G_n$. We define the pseudo inverse of $(K_n)_{n\in \N}$ by
$$ \xi_s := \inf\ens{i\ge 0 : K_i \ge s},\quad \quad (s\ge 0).$$
The profile and center-of-mass walk mentioned in (1) and (2) above are obtained by  sampling  $S$ and $\bar S$ under  $\bP_\beta(\cdot\, |\, K_{\xi_L}=L, S_{\xi_L+1}=0)$.
Therefore, deriving the scaling limit of critical-IPDSAW  requires to
prove the convergence (as $L\to \infty$) of the joint law of $S$ and
$\bar S$ rescaled in time by $L^{2/3}$ and in space by $L^{1/3}$  and
to identify  its limit. The fact that the conditioning involves the
geometric area enclosed by $S$ brings us to decompose each path into
excursions (away from the origin) simply because the geometric  area
of the whole trajectory is the sum of the modulus of the algebraic
area enclosed by each excursion.  Moreover, the area enclosed by an
excursion is a heavy tailed random variable (see e.g. Lemma \ref{eq:asympataustau}).
Therefore one can reconstruct a fraction of the whole trajectory arbitrary close to $1$ by considering finitely many excursions enclosing large areas. This is the reason why Theorem \ref{thm:ab} is a cornerstone of the proof of the main result in \cite{CarPet17a}.

The last subtlety that we must take into account comes from the fact that the excursions of $S$ (whose increments are defined  in \eqref{defincre}) do not necessarily start from the origin. 
This is the reason why we have refined  the definition of excursions by introducing another slightly different sequence of stopping times, i.e., 
$\tau_0=0$ and  for $j\geq 1$
\be{deftauseq}
 \tau_j:=\inf\ens{i>\tau_{j-1} : S_{i-1}\neq 0,\,  S_{i-1} S_i \le 0}.
 \ee

The geometric nature of the increments of $S$ yields that the excursions 
$$\big\{(V_{\tau_{j-1}+i})_{i=0}^{\tau_j-\tau_{j-1}-1}, \ \ j\geq 2\big\}$$  
are  IID and that for $j\geq 2$ their initial point  $V_{\tau_{j-1}}$  has law  $\mu$ defined by
$$ \mu(k) = \frac{1-e^{-\frac\beta2}}{2} e^{-\frac\beta2\valabs{k}}
\un{k\neq 0} + (1-e^{-\frac\beta2}) \un{k=0}.$$

Let ${\bf P}_{\beta,\mu}$ denote the distribution of
the random walk $S$ of increments $(X_i)_{i\geq 1 }$ defined in \eqref{defincre} and such that $S_0$ has law  $\mu$. Let also $\sigma_\beta$ be the variance of the random variable $X_1$
of law ${\bf P}_\beta$ (or ${\bf P}_{\beta,\mu}$). 
\begin{atheo}
  \label{thm:b}
As $L\to \infty$ and when sampled under ${\bf P}_{\beta,\mu}(. \mid \tau + G_{\tau -1}= L)$, 
$$ \unsur{\sigma_\beta L^{1/3}} (\valabs{S_{\xi_{sL}}},\bar{S}_{\xi_{sL}}; 0\le
  s \le 1)\  \text{converges in distribution  to}\  
 (\cE_{a_\cE(s)},B_{a_\cE(s)}; 0\le s\le 1).$$
\end{atheo}
There are three minor issues to settle to turn the proof of Theorem
\ref{thm:ab} into that of Theorem \ref{thm:b}. First, in Theorem
\ref{thm:b}  the law of $S_0$ is $\mu$ rather than $\delta_0$. Second,
by definition of $\tau$ the random walk in Theorem \ref{thm:b}  may
stick to the origin for a while before starting its excursion and this
is not the case in Theorem \ref{thm:ab}. Third, the conditionings are
different.
\iflong
Section 7 contains a detailed proof of Theorem \ref{thm:b} :
\fi
the geometric nature of the increments (of $S$ sampled under ${\bf P}_{\mu,\beta}$) is the key to overstep these three difficulties. 
For instance, it yields that $(S_i)_{i=0}^{\tau-1}$ has the same law when it is  conditioned by $\ens{\tau + G_{\tau-1}=L}$ or by $\ens{\tau + G_{\tau-1}=L, S_{\tau}=0}$ (this last conditioning
being much easier to relate with that of Theorem \ref{thm:ab}, i.e.,  
 $\ens{A_{\tilde \tau}=L,S_{\tilde \tau}=0}$). 

\section{Sketch of the proof of Theorem \ref{thm:a}}\label{ps}

 Proving the first convergence in Theorem \ref{thm:a} requires to establish the tightness and the convergence of the finite dimensional marginals of ($\frac{1}{\sigma \sqrt{\tilde \tau}} S_{s \tilde \tau})_{s\in [0,1]} $ when $S$ is sampled under  $\bbP_0(\cdot\, |\, A_{\tilde \tau}=N, S_{\tilde \tau}=0)$. For every $d\in \N$ and $t=(t_1,\dots,t_d)$ (with $0\leq t_1<\dots<t_d\leq 1$), we define
 \be{defsnt}
 {\bf S}_{\lfloor Nt\rfloor}:= (S_{\lfloor Nt_1\rfloor}, \dots, S_{\lfloor Nt_d\rfloor}).
 \ee
 For the convergence,  we  proceed as follows:
 \smallskip

\begin{enumerate}
\item We state and prove \emph{Proposition \ref{thm:cvloiplusapproxregegallclt}} which generalizes a method introduced in \cite{MR1340834} by Davis and McDonald to derive local central limit theorems 
(LCLT) in a non-standard manner. To be more specific, we consider a sequence of random vectors $({\bf X}_n)_{n\geq 1}$ of law $\PP$ taking values in  a countable subset $\mathcal{X}$ of 
$\R^d$. This method can be implemented to estimate  $\PP({\bf X}_n=y)$ for $n$ large and $y \in \mathcal{X}$ but it  requires a good control on the spatial gradient of  $\PP({\bf X}_n=y)$ and also a convergence in distribution of 
$({\bf X}_n)_{n\in \N}$ towards a continuous random vector with a sufficiently smooth density.

\item For every $d\in \N$ and $t=(t_1,\dots,t_d)$ (with $0\leq t_1<\dots<t_d\leq 1$), we apply Proposition \ref{thm:cvloiplusapproxregegallclt} to prove  \emph{Proposition 
\ref{pro:lltbridge}}  which gives a LCLT for the joint distribution of $A_N$ and ${\bf S}_{[Nt]}$
with $S$ sampled from  $\bbP_0(\cdot\, |\, \tilde \tau=N,S_N=0)$.  Checking that the hypothesis of Proposition \ref{thm:cvloiplusapproxregegallclt} are satisfied requires, for every $\delta>0$, to 
bound from above the spatial gradient of quantities like
$$\bbP_0(A_N=a, S_N=y, \tilde \tau>N), $$
uniformly in $ y\geq \delta \sqrt{N}$ and in  $a\in \N$.
To that aim, we use a method introduced in \cite{DKW13}, which consists in relaxing in the conditioning the constraint that $S$  remains positive 
on some small windows on the left of time $N$. On such windows $S$ has no sign constraint anymore and therefore 
classical LCLT are available for the area, the position and their respective gradients. 

\item  We  \emph{rewrite the LCLT} in Proposition \ref{pro:lltbridge} but with the event  $A_N=a$ in the conditioning.  
Then, we turn this LCLT into a convergence of the finite dimensional marginals of  $(\frac{1}{\sigma \sqrt{\tilde \tau}} S_{s\tilde \tau})_{s\in [0,1]}$  by 
summing over the length of the excursion $N$, over the positions of $S$ at times $t_1N, \dots, t_d N$, and by performing a standard Riemann approximation.

\end{enumerate} 
\medskip

The  \emph{tightness} of ($\frac{1}{\sigma \sqrt{\tilde \tau}} S_{s\tilde \tau})_{s\in [0,1]}$ is stated in Lemma \ref{lem:tight} 
and proven in Section \ref{subsec:tightness}. By time reversibility and since we show that $(\frac{\tilde \tau}{L^{2/3}})_{L\in \N}$ is a tight sequence, it suffices 
to prove the tightness of the process indexed by $s\in [0,\frac{1}{2}]$, with $S$ sampled from  $\bbP_0(\cdot\, |\, \tilde \tau=N, A_N=L,S_N=0)$ and with $N\sim L^{2/3}$.
This last point is obtained with the help of a result from \cite{B76} which gives the convergence of $(\frac{1}{\sqrt{N}} S_{\lfloor sN\rfloor})_{s\in [0,1]}$ towards the Brownian meander 
when $S$ is sampled from $\bbP_0(\cdot\, |\, \tilde \tau>N)$.

%
%

\vspace{0cm}

\section{A generalization of Davis and McDonald's proof of LCLT}
\label{sec:macdonald}

Let us first write a multidimensional version of a result of
\cite{MR1340834}. Let $(e_1,\dots,e_d)$ be the canonical basis of $\R^d$ 
and let  $\nabla_i$ be the discrete gradient in direction $i$, that is
 $\nabla_i f(x) = f(x+e_i) -f(x)$ for $f:\Z^d\to \R$ and $x\in \Z^d$.
 

In the statement of Proposition \ref{thm:cvloiplusapproxregegallclt} below,  
$G_1,G_2$ and $G_3$ are generic functions 
that satisfy $\lim_{\eta \to 0^+} G_j(\eta)=0$ for $j\in \{1,2,3\}$. {\color{black} We also introduce a subset $A(\eta)$ of $\R^d$ and we refer to 
\eqref{defaeta} below for an example of such a subset in dimension $3$.}

%
 
 \begin{proposition}\label{thm:cvloiplusapproxregegallclt}
  Let $({\bf X}_n)_{n\in \N}$ be a sequence of random vectors in  $\Z^d$ of law $\mathbb{P}$. Given $(r_1, \ldots, r_d)
  \in(0,+\infty)^d$ let us  consider the scaling operator $\Lambda_n : \R^d \to
  \R^d$ defined by  $\Lambda_n(x_1, \ldots, x_d) = (\frac{x_1}{n^{r_1}},
  \ldots, \frac{x_d}{n^{r_d}})$. Assume that
 $\Lambda_n({\bf X}_n)$
converges in distribution to a random vector ${\bf Z}$ with a Lipschitz continuous density
  $\phi$. 

Assume that for any $\eta>0$ there exist a family of real numbers
$(p_\eta(x,n),x\in\Z^d, n\ge 1)$, a set $\cA(\eta)\subset \R^d$
  and constants $C_\eta>0,\, D_\eta>0$ such that, with $r:=r_1 +\cdots+r_d$,
  \smallskip
 \begin{enumerate}[(i)]
  \item $$
    \valabs{\nabla_i p_\eta(x,n)} \le \, \frac{C_\eta}{ n^{r+r_i}}
    \qquad\qquad \text{for}\  i\in\ens{1,\ldots,d}, x\in \Z^d, n\in \N.
$$ 
\item  $$ \sup_{x:\Lambda_n(x) \in \cA(\eta)} n^r \valabs{\probb{{\bf X}_n=x}
    -p_\eta(x,n)} \le G_1(\eta) \qquad  \text{for}\   n  \ \text{large enough}.$$
\item $$  \sup_{x:\Lambda_n(x)\not\in \cA(\eta)} n^r \probb{{\bf X}_n=x} \le G_2(\eta) \ \quad   \text{for}\  n\  \text{large enough}.$$
\item  $$ \sup_{x \not\in \cA(\eta)} \phi(x)\leq G_3(\eta) \ \text{ and for all} \ x\in \cA(\eta), \ \ens{y: \norme{y-x}<
    D_\eta} \subset \cA(\eta/2).$$

  \end{enumerate}
  \medskip
Then we have the LCLT
\begin{equation}
  \lim_{n\to +\infty} \sup_x \valabs{n^r\probb{{\bf X}_n = x} - \phi(\Lambda_n(x))} =0.
\end{equation}
\end{proposition}

%
%
%

%
\begin{remark}\label{rem1}
Since  $\phi$ is a Lipschitz continuous density, $\phi$ is necessarily bounded on $\R^d$. 
\end{remark}

\begin{proof}
{\color{black} For $x=(x_1,\dots,x_d)\in \R^d$ we let $\norme{x}:=\max\{|x_i|, i\leq d\}$ be its infinity norm}. We will denote by  $\norme{\phi}_{Lip}$ the  Lipschitz constant of $\phi$.
  For $x\in \Z^d$ and $n\in \N$  we set 
 $$\Rect(x):= \prod_{i=1}^d [x_i,x_{i}+1] \quad \text{and}\quad 
   h(x,n):= \probb{{\bf Z} \in \Lambda_n(\Rect(x))},$$ where we abuse notation and write  $\Lambda_n(\Rect(x))$ 
for the image of 
$\Rect(x)$ by $\Lambda_n$.
Using that $\phi $ is the density of ${\bf Z}$  we 
write, for $x\in \Z^d$ and $n\in \N$
\begin{align}\label{vci}
 \valabs{n^r h(x,n) - \phi(\Lambda_n(x))} &= \bigg | \int_{\Rect(x)}
    \phi(\Lambda_n(v) -\phi(\Lambda_n(x))\, dv\bigg | \\
\nonumber  & \le \norme{\phi}_{Lip}\int_{\Rect(x)} \norme{\Lambda_n(v)
    -\Lambda_n(x)}\, dv \le \norme{\phi}_{Lip} C \frac{1}{n^{g_r}},
\end{align}
{\color{black} where we set $g_r:=\inf\{r_i;\,  i\leq d\}.$} 
Therefore, the proof will be complete once we show that
\begin{equation} \label{eq:sufhxn}
\lim_{n\to +\infty} n^r \sup_{x\in \Z^d} \valabs{\probb{{\bf X}_n=x} - h(x,n)} =0.
\end{equation}
Assumption (i) combined with the fact that for $x\in \Z^d$, $n\in \N$  and $i\in \{1,\dots,d\}$,
\begin{equation}
  n^r \valabs{\nabla_i h(x,n)} = \bigg | \int_{\Rect(x)}
    \phi(\Lambda_n(v) - \phi(\Lambda_n(v+e_i))\, dv \bigg | \le
  \norme{\phi}_{Lip} \frac{1}{n^{r_i}},
\end{equation}
guarantees that $\ell_\eta(x,n) := n^r(p_\eta(x,n)-h(x,n))$ has
the smoothness
\begin{equation}\label{smoothelleta}
  \valabs{\nabla_i \ell_\eta(x,n)}  \le
  (\norme{\phi}_{Lip}+ C_\eta)\,  \frac{1}{n^{r_i}} =: C_\eta'\,  \frac{1}{n^{r_i}}.
\end{equation}
We let  $\ell(x,n) :=n^r( \probb{{\bf X}_n=x} - h(x,n))$. Observe that the convergence in distribution, combined with the
continuity of the density of ${\bf Z}$
 implies the uniform convergence of
distribution functions and therefore that for any 
 $\alpha>0$
\begin{equation}
 \lim_{n\to \infty}  \sup_{x\in \Z^d} \valabs{\probb{\norme{\Lambda_n({\bf X}_n)
        -\Lambda_n(x)} \le \alpha} -\probb{\norme{{\bf Z} -\Lambda_n(x)} \le \alpha}}=0.
\end{equation}
As a consequence, for any $\alpha>0$  we set
\be{defbe2}
 \beta_n(\alpha) :=\frac{1}{n^r} \sup_x \valabs{ \sum_{y : \norme{\Lambda_n(y)
      -\Lambda_n(x)} \le \alpha} \ell(y,n)},
\ee
and $\lim_{n\to \infty} \beta_n(\alpha)=0$.
%


At this stage, we assume that \eqref{eq:sufhxn} is not satisfied. Then, there exist
$\epsilon >0$ and a non-decreasing  sequence  $(n_k)_{k\in \N}$ satisfying $\lim_{k\to \infty} n_k= +\infty$
such that 
\begin{equation}
  \sup_x \valabs{\ell(x,n_k)} \ge \epsilon \quad \text{for every}\quad  k\ge 1.
\end{equation}
We pick $\eta=\eta(\gep)>0$ such that 
\be{defeta}
 2  G_3(\eta) + G_2(\eta) < \epsilon/2 \quad \text{and}\quad G_1(\eta) \le \epsilon/2 \quad \text{and}\quad  G_1(\eta/2) \le \epsilon/8.
\ee
Let   $n=n_k$ with $k$
large enough so that 
$\norme{\phi}_{Lip} n^{-g_r} \le G_3(\eta)$ and so that (ii) and (iii) are verified. Then, for all $x\in \Z^d$ such that $\Lambda_n(x)\not\in \cA(\eta)$ we use assumption (iv) to write

\begin{equation}
   n^r h(x,n) \le    n^r \int_{\Lambda_n(\text{Rect}(x))} |\phi(y)-\phi(\Lambda_n(x))| dy+ \phi(\Lambda_n(x)) \le 
  \frac{\norme{\phi}_{Lip}}{n^{g_r}} + G_3(\eta) \le 2 G_3(\eta),
\end{equation}
and thus with the help of (iii) and \eqref{defeta} we obtain  
\begin{equation}\label{eqell}
   \sup_{x: \Lambda_n(x)\not\in \cA(\eta)}  \valabs{\ell(x,n)}  \le 2
   G_3(\eta) + G_2(\eta)\leq \epsilon/2.
\end{equation}
By assumption, there exists 
   $x_n\in \Z^d$  such that  $\valabs{\ell(x_n,n)} \ge \epsilon$ and therefore \eqref{eqell} ensures that
   $\Lambda_n(x_n) \in \cA(\eta)$.
Assume that  $\ell(x_n,n)\ge \epsilon$ (the case $\leq-\gep$ is taken care of similarly). Since 
$\Lambda_n(x_n) \in \cA(\eta)$,  (ii) and \eqref{defeta} imply that 
\be{boundii}
\ell_\eta(x_n,n)\geq \epsilon/2.
\ee
At this stage we set
 $$ \alpha:=\min\big\{\tfrac{\epsilon}{4d C'_{\eta}}, \tfrac{D_\eta}{2}\big\},$$
  so that
 (iv) ensures us that  for any  $y\in \Z^d$ satisfying
$\norme{\Lambda_n(y) -\Lambda_n(x_n)} \le \alpha$, we have $\Lambda_n(y)
\in \cA(\eta/2)$ and thus by (ii) and \eqref{defeta}
\begin{equation}
  \valabs{\ell(y,n) -\ell_\eta(y,n)}=n^r \valabs{\probb{{\bf X}_n=y}-p_{\eta}(y,n)} \le G_1(\eta/2) < \epsilon/8.
\end{equation}
Thanks to $\ell_\eta$'s smoothness in \eqref{smoothelleta}, for any  $y$ such that
$\norme{\Lambda_n(y) -\Lambda_n(x_n)} \le \alpha$ we get
\begin{equation}\label{surlesvalabs}
  \valabs{\ell_\eta(y,n) -\ell_\eta(x_n,n)} \le C_\eta'\, d\,  \alpha \le \frac{\epsilon}{4},
\end{equation}
and thus (\ref{boundii}--\ref{surlesvalabs}) yield   $\ell(y,n) \ge \epsilon/8$. Coming back to \eqref{defbe2}, we bound $n^r \beta_n(\alpha)$ from below by 
\begin{equation}
  \sum_{y : \norme{\Lambda_n(y) -\Lambda_n(x_n)} \le \alpha} 
  \ell (y,n) \ge
  \frac{\epsilon}{8}\,\,   \big | \{y : \norme{\Lambda_n(y) -\Lambda_n(x_n)} \le \alpha\}\big |
  \ge \frac{\epsilon}{8}\, (2\alpha)^d  n^r,
\end{equation}
which contradicts the fact that $\lim_{n\to \infty}\beta_n(\alpha)=0$ and therefore 
proves \eqref{eq:sufhxn}.

\end{proof}

\section{Proof of Theorem \ref{thm:a}}
\label{sec:preuvea}

We derive a local central limit theorem (LCLT) for the excursion, that is we
condition now on the event $\ens{\tilde \tau=N,S_N=0}$.  Set $t=(t_1,\ldots,t_d)$
with $0< t_1 < \cdots < t_d < 1$. {\color{black} Remember the definition of $\snt$ in \eqref{defsnt}} and let $\phi_{t}(a,x)$ be the density of 
\be{defbe}
 \Big(A_1(e)= \int_0^1 e_s\, ds, e_{t_1},\cdots,e_{t_d}\Big),
 \ee
with $e$ the standard Brownian excursion. 
  \begin{proposition}\label{pro:lltbridge}  
 With $r= \frac{d}{2} + \frac32$, 
\begin{equation}
\lim_{N\to +\infty} \sup_{(a,x)\in \Z^{d+1}} \valabs{N^r \prob{A_N=a,
  \snt=x \mid \tilde \tau =N,S_N=0} 
-\frac{1}{\sigma^{r}} \phi_t\etp{\frac{a}{\sigma
    N^{3/2}}, \frac{x}{\sigma N^{1/2}}}}=0.
\end{equation}
\end{proposition}
The next step is to bring the condition on the area enclosed in the conditioning (i.e., the event $\{A_N=a\}$). It is the reason why we
derived a LCLT for the joint process (area at time $N$, position at times $tN$).  Recall \eqref{defEE}.


\begin{proposition}\label{pro:cvconjointetaurw}
Let $t=(t_1,\ldots,t_d)$
with $0< t_1 < \cdots < t_d < 1$. Then, as  $L\to \infty$ and when sampled under
$\prob{\cdot \mid A_{\tilde \tau} =L, S_{\tilde \tau}=0}$  
$$\displaystyle\etp{\frac{\tilde \tau}{ L^{2/3}},
  \frac{\mathbf{S}_{\floor{t\tilde \tau}}}{\sigma\sqrt{\tilde \tau}}}\quad \text{
converges in distribution to}\quad 
\displaystyle \etp{R_{\cE}, \unsur{\sqrt{R_{\cE}}}\cE(t_1 R_{\cE}),\dots, \unsur{\sqrt{R_{\cE}}}\cE(t_d R_{\cE})}.$$
  
\end{proposition}

By establishing
tightness in Section \ref{subsec:tightness}, this convergence is lifted to the process level and 
proves the first convergence in Theorem \ref{thm:a}.
By applying an ad-hoc time change we finally obtain the second convergence in Theorem \ref{thm:a} in
Section \ref{sub:timechange}.

\subsection{Proof of Proposition~\ref{pro:lltbridge}}\label{subsec:a:lltexc}

We shall restrict ourselves to $d=2$ since all the technical
ingredients are present there. Therefore, we consider $t=(t_1,t_2)$ with 
$0< t_1< t_2 < 1$.  We are going to apply
Proposition~\ref{thm:cvloiplusapproxregegallclt}. For that purpose we define, for $N\in \N$,  the operator $\Lambda_N$ as 
\be{defl}
\Lambda_N(a,x)=(N^{-3/2}a,N^{-1/2}x) \quad \text{for} \quad a\in \R, x=(x_1,x_2)\in \R^2.
\ee
We recall \eqref{defsnt} and \eqref{defbe} and
we consider  $Y_N$ distributed as  $(A_N,\mathbf S_{\floor{tN}})$ under $\prob{. \mid
  \tilde \tau =N, S_N=0}$. A consequence of 
\cite[Corollary 2.5]{CarCha13} is that $\unsur{\sigma}\Lambda_N(Y_N)$ converges in
distribution towards $(A_1(e),e(t_1), e(t_2))$. The fact that
$\phi_t(a,x)$, the density of  $(A_1(e), e(t_1), e(t_2))$, is
Lipschitz continuous is the object of
\iflong
Lemma A.1.
\fi Without loss of generality we assume  $\sigma=1$.
\medskip

For every $\eta>0$ we set  
\be{defaeta}
\cA(\eta) := \ens{(a,x) \in \R^{3} \colon\, 
  x_1\ge \eta, x_2 \ge \eta}.
\ee   
For $(k,l)\in \N^2$ such that $k+1\leq l-1$ we denote by $\cW_{k,l}$ the set of path which remain positive 
between times $k$ and $l$, i.e., 
\be{defwkl}
\cW_{k,l}=\big\{S \in \Z^{\N_0} \colon\, S_i>0, \forall i\in \{k+1,\dots,l-1\} \big\}.
\ee
For   $N\in \N$ and $\eta>0$  we set 
$M_N:=\floor{\eta^3 N}$ and to simplify notations, we will omit the $N$ subscript in $M_N$. We set also for  $x\in \Z$
$$ Q_N(x) = \big\{z \in \Z : \valabs{x-z} \le \tfrac{1}{2} \eta \sqrt{N}\big\}.$$
We recall \eqref{defStiS} and also that $\PP_x$ the law of $(x+S_n)_{n\geq 0}$ for $x\in \Z$,. 
Then, for  $a\in \N$ and $x_1, x_2\in \N$ we set 
 \begin{align}\label{eq:defalphaeta}
   \alpha_\eta(a,0,x_2,N) &:=
  \PP_{0}\big(S\in \cW_{0,N-M_N},\,  A_N=a, \, S_N=x_2, \, S_{N-M}\in Q_N(x_2)\big),\\
\nonumber    \alpha_\eta(a,x_1,0,N) &:=
  \PP_{x_1}\big(S\in \cW_{M,N},\, A_N=a, \, S_N=0, \, S_{M}\in Q_N(x_1)\big), \\
\nonumber    \alpha_\eta(a,x_1,x_2,N) &:=
  \PP_{x_1} \big(S\in \cW_{M,N-M}, \, A_N=a,\, S_N=x_2,\,  S_{M}\in Q_N(x_1), \, S_{N-M}\in Q_N(x_2)\big).
\end{align}
\smallskip
For $a,k\in \N$  we introduce 
\be{defGa}
\mathfrak{G}_{k,a}:=\Big\{(a_i)_{i=1}^k\in \N^k\colon  \sum_{i=1}^k a_i=a\Big\} \quad \text{and} \quad  \mathfrak{G}^-_{k,a}:=\Big\{(a_i)_{i=1}^k\in \N^k\colon \sum_{i=1}^k a_i\leq a\Big\},
\ee 
and for $\eta>0$, $N\in \N$, $a\in \N$ and $x_1, x_2\in \N_{0}$ we set
 \begin{align}\label{defpeta}
  p_\eta(a,&x_1,x_2,t_1,t_2,N)\\
  \nonumber  &  = \unsur{\prob{\tilde \tau =N,S_N=0}} \sum_{(a_1,a_2,a_3)\in \mathfrak{G}_{3,a} }
  \alpha_\eta(a_1,0,x_1,r_1) 
\alpha_\eta(a_2,x_1,x_2,r_2) \, \alpha_\eta(a_3,x_2,0,r_3).
\end{align}
with $r_1=\floor{Nt_1}, r_2= \floor{Nt_2}-\floor{Nt_1}$ and $r_3=N-r_2$. 
\subsubsection{\bf Useful bounds}\label{ub}
In the present section we collect and prove some upper-bounds that will be important to verify assumptions (i--iv) of Proposition~\ref{thm:cvloiplusapproxregegallclt}.

Since the random walk $S$ is centered with
finite variance, the upper tail of $\tilde \tau$ can be bounded above, first for a random walk starting from the origin (see  \cite[Lemma 5.1.8]{LL10}) and also when the starting point 
of the random walk is positive but "small" (see e.g. \cite[(3.9)]{DKW13}). Thus, there exist $C_1>0$ and $\gep_1:(0,\infty)\to (0,\infty)$ satisfying $\lim_{\eta\to 0} \gep_1(\eta)=0$ 
such that for every $N\in \N$  
\begin{equation}\label{eq:simtau}
  \prob{\tilde \tau > N} \leq \frac{C_1}{ N^{1/2}} \quad \text{and}\quad  \sup_{0\le x \le \eta
    \sqrt{N}}\probx{\tilde \tau >N} 
  \le \gep_1(\eta).
\end{equation}
A classical Gaussian LCLT (see e.g. \cite[Theorem 2.3.10]{LL10}) and \cite[Proposition 2.3]{CD08}
ensure the existence of $C_2, C_3>0$  such that for every $N\in \N$ 
\begin{equation}\label{eq:simtau2}
 \sup_{x,y \in \Z} \probx{S_N=y} \le \frac{C_2}{N^{1/2}} \quad \text{and}\quad  \sup_{a,x,y\in \Z}\,  \probx{A_N=a, S_{N}=y} \le \frac{C_3}{N^{2}}.
\end{equation}
Combining the preceding inequalities, we obtain that there exist $C_4, C_5>0$ such that for every $N\in \N$,
\begin{equation}\label{eq:crudeboudsnxtau}
 \sup_{x\in \N}\,  \prob{S_N=x,\tilde \tau >N} \le \frac{C_4}{N} \quad \text{and}\quad  \sup_{a,x\in \N}\, \prob{A_N=a, S_N=x,\tilde \tau >N} \le \frac{C_5}{N^{5/2}}.
\end{equation}
We display here the  proof of the second  inequality in \eqref{eq:crudeboudsnxtau} since the proof of the first inequality is similar and easier. We apply Markov's property at time $K=\floor{\frac{N}{2}}$ and we recall \eqref{eq:simtau} and \eqref{eq:simtau2} to write
\begin{align}\label{airets}
  \nonumber  \PP_0 \big(A_N=a, \,  S_N=x,\,  &\tilde \tau >N\big) \\
\nonumber   &\le \sum_{z,a_1+a_2=a}
                                \prob{A_K=a_1,S_K=z, \tilde \tau > K}
    \probz{A_{N-K}=a_2,S_{N-K}=x}\\
 &\le  \frac{C_3}{(N-K)^{2}} \ \prob{\tilde \tau >K} \le \frac{C_5}{N^{5/2}}.
\end{align}

%
{\color{black} We can provide the counterparts bounds of} \eqref{eq:crudeboudsnxtau} when the starting point of the random walk is positive but "small". 
In other words, there exists $\gep_2:(0,\infty)\to (0,\infty)$ such that $\lim_{\eta\to 0}\gep_2(\eta)=0$ and such that   for every $N\in \N$
\begin{equation} \label{eq:tsnxpetit}
  \sup_{\genfrac{}{}{0pt}{}{ 0\le x \le \eta \sqrt{N}}{ y\in\N}} \probx{S_N=y, \tilde\tau >N} \le
\frac{\gep_2(\eta)}{ N^{1/2}}  \quad \text{and}\quad   \sup_{\genfrac{}{}{0pt}{}{0\le x \le \eta\sqrt{N}}{a,y\in \N}} \probx{A_N=a, S_N=y,\tilde \tau >N} \le \frac{\gep_2(\eta)}{ N^{2}}.
\end{equation}
Again, the proof of both inequalities above being very similar we only display the proof of the first inequality.
By Markov's property at time $K= \floor{\frac{N}2}$ and by recalling (\ref{eq:simtau}--\ref{eq:simtau2}) we get
\begin{align*}
  \probx{S_N=y, \tilde \tau >N} 
                        &\le \sum_{z\in \N} \probx{S_K=z, \tilde \tau >K} \probz{S_{N-K}=y} \\
                        &\le \frac{C_2}{ (N-K)^{1/2}}\,  \probx{\tilde \tau >K} \le \frac{C_2}{ (N-K)^{1/2}}\,  \gep_1(\sqrt{2}\eta)  \le \frac{ \gep_2(\eta)}{N^{1/2}}. 
\end{align*}

We also state the counterparts  inequalities of \eqref{eq:tsnxpetit} but for a  random walk starting from the origin and with a "small" endpoint, i.e., 
 there exists $\gep_3:(0,\infty)\to (0,\infty)$ satisfying 
$\lim_{\eta\to 0}\gep_3(\eta)=0$ and such that   for every $N\in \N$,
\begin{equation} \label{eq:e:lltboundxsmall}
  \sup_{ 0\le x \le \eta \sqrt{N}} \prob{S_N=x, \tilde\tau >N} \le
\frac{\gep_3(\eta)}{ N}  \quad \text{and}\quad   \sup_{\genfrac{}{}{0pt}{}{0\le x \le \eta\sqrt{N}}{a\in \N}} \prob{A_N=a, S_N=x,\tilde \tau >N} \le \frac{\gep_3(\eta)}{ N^{5/2}}.
\end{equation}
We only prove the first inequality in \eqref{eq:e:lltboundxsmall} since the proof of the second inequality is again completely similar.
By Markov's property at time $K= \floor{\frac{N}2}$ and with the help of \eqref{eq:crudeboudsnxtau} we get
\begin{align}\label{cii}
 \nonumber  \prob{S_N=x , \tilde  \tau >N } &= \sum_{z\in \N}
                            \prob{S_K=z, \tilde\tau>K}\, \probz{S_{N-K}=x, \tilde\tau
                            >N-K} \\
                          &\le \frac{C_4}{K}  \sum_{z\in \N}
                                        \probz{S_{N-K}=x, \tilde\tau >N-K}.                                                               
\end{align}
We now use  \emph{time reversal} and we will use it again several times in the present paper. For this  we consider the random walk $S':=(S'_n)_{n\in \N}$
with increments $X'_i=-X_i$ and we set $\tilde \tau'=\inf\ens{n\ge 1: S'_n \le
  0}$. We note that the increments of $S'$ also satisfy \eqref{defhyprw}. Since
\be{tr1}
 \probz{S_{N-K}=x, \tilde \tau > N-K} = \probx{S'_{N-K}=z,\tilde \tau'>N-K},
 \ee
we use \eqref{eq:simtau} with $\tilde \tau'$ and we rewrite \eqref{cii} as
\begin{align*}
  \prob{S_N=x , \tilde \tau >N }&\le \frac{C_4}{K} \sum_{z\in \N}  \probx{S'_{N-K}=z,\tilde \tau'>N-K}\\
&=
\frac{C_4}{K} \probx{\tilde \tau'>N-K} \le \frac{C_4 \, \gep_1(\sqrt{2}\eta)}{K}= \frac{\gep_3(\eta)}{N}.
\end{align*}

Using the extra moments assumption \eqref{defhyprw}, we obtain a standard Gaussian local limit
theorem for discrete gradients as in \cite[Theorem 2.3.6]{LL10}: there exists $C_6>0$ such that 
for every $N\in \N$,
\begin{align}\label{eq:smoothnessstandartllt}
\sup_{y\in \Z}\valabs{\nabla_y\prob{S_N=y} }& \le C_6 N^{-1}.
\end{align}
\iflong
Let us generalize this bound:
\fi
there exist $C_7, C_8>0$ such that 
\begin{align}\label{eq:smoothnessstandartllt22}
\sup_{y,a\in \Z}\nonumber   \valabs{\nabla_y\prob{S_N=y,\, A_N=a} }& \le C_7 N^{-5/2}\, ,\\
\sup_{y,a\in \Z} \valabs{\nabla_a\prob{S_N=y,\, A_N=a} }& \le C_8 N^{-\text{7/2}} \,. 
\end{align}

\iflong
\begin{proof}
Let $f$ be the density of $(B_1,\int_0^1 B_s ds)$ with $B$ a standard
Brownian motion. Let
\begin{equation}
  p_n(k,a) = \prob{S_n=k,A_n=a}\,,\quad \bar{p}_n(k,a) = \unsur{n^2
    \sigma^2} f\etp{\frac{k}{ \sigma \sqrt{n}},\frac{a}{\sigma
      n^{3/2}}}\, \quad(k,a \in \Z)\,.
\end{equation}
Let $\nabla_k p_n(k,a)=p_n(k+1,a) -p_n(k,a)$ and $\nabla_k
\bar{p}_n(k,a)=\bar{p}_n(k+1,a) -\bar{p}_n(k,a)$. We proved in \cite[Equation
(7.11)]{CP15} that
\begin{equation}
  \sup_{k,a\in\Z} n^3 \valabs{p_n(k,a) - \bar{p}_n(k,a)} < +\infty.
\end{equation}
We are going to prove that
\begin{equation}\label{eq:sup_k-ainz-n72}
  \sup_{k,a\in\Z} n^{7/2} \valabs{\nabla_kp_n(k,a) - \nabla_k\bar{p}_n(k,a)} < +\infty.
\end{equation}
This easily implies the first assertion in
\eqref{eq:smoothnessstandartllt22}.

We established in  \cite[Equation
(7.23)]{CP15} that for a $\eta>0$ small enough there exists $\beta >0$
such that
\begin{equation}
  p_n(k,a) - \bar{p}_n(k,a) = O(e^{-\beta n}) + \unsur{4\pi^2 n^2}
  \int_{\ens{\valabs{\theta} \le \eta \sqrt{n}}} e^{-i
    z.\theta-\undemi \Gamma\theta.\theta}\etp{e^{g(n,\theta)}-1}\, d\theta\,,
\end{equation}
with $\Gamma$ the covariance matrix of the gaussian vector
$(B_1,\int_0^1 B_s\, ds)$ and $g$ a
function such that for a constant $C>0$, forall $\valabs{\theta} \le \eta \sqrt{n}$
: $\valabs{g(n,\theta)}\le
C\max\etp{\Gamma\theta.\theta,\valabs{\theta}^2/n}$. Hence we have
\begin{align*}
  \nabla_kp_n(k,a) - \nabla_k\bar{p}_n(k,a) &= O(e^{-\beta n}) \\
  & + \unsur{4\pi^2 n^2}
  \int_{\ens{\valabs{\theta} \le \eta \sqrt{n}}} e^{-i
    z.\theta-\undemi
    \Gamma\theta.\theta}\etp{e^{g(n,\theta)}-1}(e^{i\frac{\theta_1}{\sqrt{n}}} -1)\, d\theta\,.
\end{align*}
Therefore when we take modulus inside the integral we just multiply
the integrand by the factor
\begin{equation}
  \valabs{e^{i\frac{\theta_1}{\sqrt{n}}} -1}\le \frac{C}{\sqrt{n}}
\end{equation}
and this explains how we get the extra $n^{1/2}$ in
\eqref{eq:sup_k-ainz-n72}. It should be now clear that when we 
prove the second equation in \eqref{eq:smoothnessstandartllt22} we
shall get an extra $n^{3/2}$.
\end{proof}

\fi

\subsubsection{{\bf Proof of (i)}}
In the framework of Proposition~\ref{thm:cvloiplusapproxregegallclt}, we are working with $r=5/2$.
We recall \eqref{defpeta}. We need to prove the regularity of $p_\eta$ in $a$, $x_1$ and $x_2$. Thus, we must show that there exists $C_\eta>0$
such that for every $(a,x)\in \N\times\N^2$ and $N\in \N$ we have 
\begin{align}\label{eqnabla}
\big|\nabla_{a}p_\eta(a,x,t_1,t_2,N)\big |\leq \frac{C_\eta}{N^{4}}\, ,\quad 
\big|\nabla_{x_1}p_\eta(a,x,t_1,t_2,N)\big |\leq \frac{C_\eta}{N^3}\,, \quad\big|\nabla_{x_2}p_\eta(a,x,t_1,t_2,N)\big |\leq \frac{C_\eta}{N^{3}}.
\end{align}
\smallskip

As in \eqref{tr1} we consider the random walk  $S'$  and we let $p'_\eta$ be defined as in  (\ref{eq:defalphaeta}--\ref{defpeta}) but with 
 $S'$ instead of $S$. Then,  by \emph{time reversal} we have 
 \be{tr}
 p_\eta(a,x_1,x_2,t_1,t_2,N)=p'_\eta(a,x_2,x_1,1-t_2,1-t_1,N).
 \ee
 Since the increments of $S'$ also satisfy \eqref{defhyprw},  we conclude that regularity of $p_\eta$ in $x_1$ implies its regularity in $x_2$
 (i.e., the second inequality in \eqref{eqnabla} implies the third inequality).

Let us prove the second inequality in \eqref{eqnabla}. For simplicity we set $u_N:=\prob{\tilde \tau =N,S_N=0}$ and we use 
\cite[Theorem 6]{VatWac09}, or \cite[(4.5)]{CarCha13} to assert that
\begin{equation}\label{eq:thm6devatutinwachtel}
 u_N=C N^{-3/2} (1+o(1)).
\end{equation}
As a consequence, the second inequality in \eqref{eqnabla} will be proven once we show that 
$u_N \big|\nabla_{x_1}p_\eta(a,x,t_1,t_2,N)\big |\leq \frac{C_\eta}{N^{9/2}}$. We note that 
\begin{align}\label{eqrewr}
\nonumber u_N \nabla_{x_1} p_\eta(a,x,t_1,t_2,N)=\hspace{-.2cm} \sum_{(a_1,a_2,a_3)\in \mathfrak{G}_{3,a}}&
 \nabla_{x_1} \alpha_\eta(a_1,0,x_1,r_1)\,   \alpha_\eta(a_2,x_1+1,x_2,r_2) \, \alpha_\eta(a_3,x_2,0,r_3)\\
 &+\alpha_\eta(a_1,0,x_1,r_1)\,    \nabla_{x_1} \alpha_\eta(a_2,x_1,x_2,r_2) \, \alpha_\eta(a_3,x_2,0,r_3).
\end{align}
We will prove that there exist $c_{1,\eta}, c_{2,\eta}>0$ such that 
for every $a\in \N$, $(x_1,x_2)\in \N$ and $N\in \N$ we have 
\be{boundalp}
\big|\nabla_{x_1} \alpha_\eta(a,0,x_1,N)\big |\leq \frac{c_{1,\eta}}{N^3} \quad \text{and} \quad \big| \nabla_{x_1} \alpha_\eta(a,x_1,x_2,N)\big |\leq \frac{c_{2,\eta}}{N^{5/2}}.
\ee
Let us begin with the first inequality in \eqref{boundalp}. Recall that $M=\floor{\eta^3 N}$  and use Markov's property at time $N-M$ to write
\begin{align}\label{nabx1}
\nonumber \big|\nabla_{x_1} \alpha_\eta(a,0,x_1,N)\big | &\leq \sum_{ b=1}^{ a} \sum_{z\in Q_N(x_1)} \PP_0(A_{N-M}=b,S_{N-M}=z, \tilde \tau>N-M) \\
\nonumber & \hspace{5cm}  \times \big|\nabla_{x_1} \PP_{z}(S_M=x_1, A_M=a-b)\big |\\
\nonumber &\leq \frac{C_7}{M^{5/2}} \sum_{b=1}^{a} \sum_{z\in Q_N(x_1)} \PP_0(A_{N-M}=b,S_{N-M}=z, \tilde \tau>N-M) \\
 &\leq \frac{C_7}{M^{5/2}} \ \PP_0(\tilde \tau>N-M)\leq  \frac{C_7}{M^{5/2}}  \frac{C_1}{\sqrt{N-M}}\leq \frac{c_{1,\eta}}{N^3},
\end{align}
where we have used the first inequality in \eqref{eq:smoothnessstandartllt22} to write the second line and  \eqref{eq:simtau} in the third line.

For the second inequality in \eqref{boundalp}, by time reversal again we note that it suffices to prove the very same inequality with  $\nabla_{x_2} \alpha_\eta(a,x_1,x_2,N)$ instead of 
$\nabla_{x_1} \alpha_\eta(a,x_1,x_2,N)$ and this turns out to be easier. We apply Markov's property at time $N-M$ and obtain 
 \begin{align}\label{nabx11}
\nonumber \big|\nabla_{x_2} &\alpha_\eta(a,x_1,x_2,N)\big | \\
\nonumber &\leq \sum_{ b=1}^{ a} \sum_{z\in Q_N(x_2)} \PP_{x_1}(S_M\in Q_N(x_1), \, A_{N-M}=b,\, S_{N-M}=z, S\in \cW_{M,N-M}) \\
\nonumber & \hspace{7.5cm}\times\big|\nabla_{x_2} \PP_{z}(S_M=x_2, A_M=a-b)\big |\\
\nonumber &\leq \frac{C_7}{M^{5/2}} \,  \PP_{x_1}\big(S_M \in Q_N(x_1),  A_{N-M}\leq a, S_{N-M}\in Q_N(x_2), S\in \cW_{M,N-M} \big) \\
&\leq \frac{c_{2,\eta}}{N^{5/2}},
\end{align}
where we have used the first inequality in \eqref{eq:smoothnessstandartllt22} to write the second line.

We deduce from  \eqref{eq:defalphaeta} and from (\ref{eq:simtau}--\ref{eq:simtau2}) that there exists $c_{3,\eta}>0$ such that 
\begin{align}\label{thy1}
\nonumber \sum_{a_1=1}^{a} \alpha_\eta(a_1,0,x_1,r_1) &\leq \PP_0( \tilde \tau>r_1-M,\, S_{r_1-M}\in Q_N(x_1),\,  S_{r_1}=x_1)\\
\nonumber &\leq \sum_{z\in Q_N(x_1)} \PP_0( \tilde \tau>r_1-M,\, S_{r_1-M}=z)\,  \PP_z( S_{M}=x_1)\\
&\leq \frac{C_2}{\sqrt{M}}\   \PP_0( \tilde \tau>r_1-M)\leq \frac{C_1 C_2}{\sqrt{M} \sqrt{r_1-M}}\leq \frac{c_{3,\eta}}{N},
\end{align}
where we have used that $r_1=\lfloor t_1 N\rfloor$ (with $t_1>0$) and $M= \lfloor \eta^3 N\rfloor$ to write the last line.
Then, by \eqref{eq:defalphaeta} and \eqref{eq:simtau2} again we write that there exists $c_{4,\eta}>0$ such that 
\be{thy2}
\sum_{a_2=1}^a \alpha_\eta(a_2,x_1+1,x_2,r_2) \leq \PP_{x_1+1}(S_{r_2}=x_2)\leq \frac{C_2}{\sqrt{r_2}} \leq \frac{c_{4,\eta}}{\sqrt N}.
\ee
Finally, by using time reversibility and the same computation as \eqref{thy1} we obtain that there exists  $c_{5,\eta}>0$ such that 
\begin{align}\label{thy3}
 \sum_{a_3=1}^{a} \alpha_\eta(a_3,x_2,0,r_3) &\leq \PP_0( \tilde \tau>r_3-M, S'_{r_3-M}\in Q_N(x_2),\, S'_{r_3}=x_2)\leq \frac{c_{5,\eta}}{N}.
 \end{align} 
At this stage, it suffices to combine \eqref{eqrewr} with \eqref{boundalp} and (\ref{thy1}--\ref{thy3}) to complete the proof of the second inequality in \eqref{eqnabla}.

It remains to prove the first inequality in \eqref{eqnabla} that is the $a$-regularity of $p_\eta$. The proof is very close in spirit to that of the second inequality displayed above. For this reason we will only sketch the key points of the proof. 
 
 By \eqref{eq:thm6devatutinwachtel}  one should prove that
$u_N \big|\nabla_{a}p_\eta(a,x,t_1,t_2,N)\big |\leq \frac{C_\eta}{N^{11/2}}$. Moreover,
\begin{align}\label{eqrewr2}
\nonumber u_N \nabla_{a}p_\eta(a,x,t_1,t_2,N)= \sum_{(a_1,a_2,a_3)\in \mathfrak{G}_{3,a}}&
 \nabla_{a} \alpha_\eta(a_1,0,x_1,r_1)\,   \alpha_\eta(a_2,x_1,x_2,r_2) \, \alpha_\eta(a_3,x_2,0,r_3)\\
 \nonumber &+\alpha_\eta(a_1,0,x_1,r_1)\,    \nabla_{a} \alpha_\eta(a_2,x_1,x_2,r_2) \, \alpha_\eta(a_3,x_2,0,r_3)\\
 &+\alpha_\eta(a_1,0,x_1,r_1)\,    \alpha_\eta(a_2,x_1,x_2,r_2) \,  \nabla_{a} \alpha_\eta(a_3,x_2,0,r_3).
\end{align}
At this stage one should control the gradients, i.e.,  prove that  there exist $c_{6,\eta}, c_{7,\eta}, c_{8,\eta}>0$ such that 
for every $a\in \N$, $(x_1,x_2)\in \N$ and $N\in \N$ we have 
\begin{align}\label{boundalp2}
\big|\nabla_{a} \alpha_\eta(a,0,x_1,N)\big |\leq \frac{c_{6,\eta}}{N^4} \ \text{and} \  \big| \nabla_{a} \alpha_\eta(a,x_1,x_2,N)\big |\leq \frac{c_{7,\eta}}{N^{7/2}}
\  \text{and} \  \big| \nabla_{a} \alpha_\eta(a,x_2,0,N)\big |\leq \frac{c_{8,\eta}}{N^{4}}.
\end{align}
By time reversibility again,  in \eqref{boundalp2},  the first upper bound implies the third upper-bound. To prove the first upper bound
we follow the computation in \eqref{nabx1}  except that we replace  $\big|\nabla_{x_1} \PP_{z}(S_M=x_1, A_M=a-b)\big |$ by  $\big|\nabla_{a} \PP_{z}(S_M=x_1, A_M=a-b)\big |$ which we bound from above using the second inequality in \eqref{eq:smoothnessstandartllt22}. We conclude that indeed the first upper bound in \eqref{boundalp2} is satisfied. Similarly, we prove the second upper bound in \eqref{boundalp2} by following \eqref{nabx11} and  replacing $\big|\nabla_{x_2} \PP_{z}(S_M=x_2, A_M=a-b)\big |$ by $\big|\nabla_{a} \PP_{z}(S_M=x_2, A_M=a-b)\big |$ which we bound from above  with the second inequality  in \eqref{eq:smoothnessstandartllt22}.

At this stage we complete the proof of first inequality in \eqref{eqnabla} by combining  \eqref{eqrewr2} with \eqref{boundalp2} and (\ref{thy1}--\ref{thy3}).

\subsubsection{{\bf Proof of (ii)}}

Recall \eqref{defaeta}. Using \eqref{eq:thm6devatutinwachtel}, we need to establish that there exists $G_1:(0,\infty)\to (0,\infty)$ such that 
$\lim_{\eta \to 0} G_1(\eta)=0$ and such that for every  $\eta>0$ we have  for $N$ large enough and for  $a\in \N$ and $x_1, x_2\geq \eta \sqrt{N}$ that
\begin{multline}\label{defG} 
  \big|\prob{A_N=a,S_{\floor{Nt_1}}=x_1,S_{\floor{Nt_2}}=x_2,\tilde \tau
     =N,S_N=0}\\
   - \sum_{(a_1,a_2,a_3)\in \mathfrak{G}_{3,a}}
  \alpha_\eta(a_1,0,x_1,r_1)\,   \alpha_\eta(a_2,x_1, x_2,r_2)\,  \alpha_\eta(a_3,x_2, 0,r_3)\,\big|\le \frac{G_1(\eta)}{N^{4}}.
\end{multline}
By  Markov's property, we have
 \begin{align}\label{eq:markovdecompositionexcursion}
  \mathbb{P}_0\big(A_N=a,S_{\floor{Nt_1}}=x_1,&S_{\floor{Nt_2}}=x_2,\tilde \tau =N,S_N=0\big)
  \\ 
 \nonumber  =\sum_{(a_1,a_2,a_3)\in \mathfrak{G}_{3,a}} & \prob{A_{r_1}=a_1,S_{r_1}=x_1,\tilde \tau > r_1}
 \\
\nonumber   &\times  \PP_{x_1}\etp{A_{r_2}=a_2, S_{r_2}= x_2,\tilde \tau > r_2}  \PP_{x_2}\etp{A_{r_3}=a_3,\tilde \tau=r_3,S_{r_3}=0}.
 \end{align}
At this stage we need to prove that there exists  $\bar G_1:(0,\infty)\to (0,\infty)$ satisfying
$\lim_{\eta \to 0} \bar G_1(\eta)=0$ and  such that for every  $\eta>0$ we have  for $N$ large enough and for  $a\in \N$ and $x_1, x_2\geq \eta \sqrt{N}$ that 
\begin{align}\label{eq:diffe}
 \big |\prob{A_{N}=a,S_{N}=x_1,\tilde \tau > N}- \alpha_\eta(a,0,x_1,N)\big |& \leq  \frac{\bar G_1(\eta)}{N^{5/2}},\\
\nonumber  \big | \PP_{x_1}\etp{A_{N}=a, S_{N}= x_2,\tilde \tau > N}  - \alpha_\eta(a,x_1, x_2,N)\big | & \leq \frac{\bar G_1(\eta)}{N^{2}}, \\
\nonumber  \big | \PP_{x_2}\etp{A_{N}=a,\tilde \tau=N,S_{N}=0}- \alpha_\eta(a,x_2, 0,N) \big |& \leq \frac{\bar G_1(\eta)}{N^{5/2}}.
\end{align}
By time reversibility, in \eqref{eq:diffe}, the first inequality  implies the third inequality.
  Let us prove the first inequality in \eqref{eq:diffe}. We  write  
  \begin{align}\label{boundArr}
    &\prob{A_{N}=a,S_{N}=x_1,\tilde \tau > N}-\alpha_\eta(a,0,x_1,N) =  \ell_{1,\eta}(a,0,x_1,N) -\ell_{2,\eta}(a,0,x_1,N),
     \end{align}
 with     
   \begin{align}\label{boundArr2}
   \ell_{1,\eta}(a,0,x_1,N)&= \prob{A_N=a,S_{N-M}\notin Q_N(x_1),S_N = x_1,\tilde \tau > N}, \\
\nonumber   \ell_{2,\eta}(a,0,x_1,N)&= \prob{A_N=a, S_{N-M}\in Q_N(x_1),S_N = x_1, N-M < \tilde \tau \le N} .
    \end{align}
     
We consider first
  \begin{align}\label{breta}
 &\ell_{1,\eta}(a,0,x_1,N)\\
\nonumber   &=\hspace{-.4cm}  \sum_{z\notin Q_N(x_1), (a_1,a_2)\in \mathfrak{G}_{2,a}}
              \hspace{-.2cm}       \prob{A_{N-M}=a_1,S_{N-M}=z,\tilde \tau > N-M}\ 
                     \probz{A_M=a_2,S_M=x_1,\tilde \tau >M}.
    \end{align}
With the help of \eqref{eq:crudeboudsnxtau}, we bound the first probability  in the right hand side (r.h.s.) of \eqref{breta} by 
$C_5/(N-M)^{5/2}$. Then, we remove the condition $\tilde \tau>N-M$ from the second probability  and  we sum it on $a_2\leq a$ 
to obtain                       
\be{reta2}
 \ell_{1,\eta}(a,0,x_1,N)  \le C_5 (N-M)^{-5/2}\sum_{z\notin Q_N(x_1)} \probz{S_M=x_1}.
 \ee
We let $(B_s)_{s\geq 0}$ be a standard 1-dimensional Brownian motion.  Using time reversibility we get that 
$\sum_{z\notin Q_N(x_1)} \probz{S_M=x_1}
=  \PP_0\big(\big |S'_{\lfloor \eta^3 N\rfloor}\big | \ge \undemi \eta
    \sqrt{N} \big)$
{\color{black} and therefore, \eqref{reta2} and a standard LCLT (see e.g. \cite{LL10})} guarantee that there exists $C>0$ such that  for $N$ large enough 
\be{bfreta1}
 \ell_{1,\eta}(a,0,x_1,N) \leq \frac{C}{N^{5/2}}\  \PP_0\Big(|B_1|\geq \tfrac{1}{2\sqrt \eta}\Big).
 \ee
We proceed similarly to bound $\ell_{2,\eta}(a,0,x_1,N)$ from above.  Thus, we apply Markov's property at time 
$N-M$  and write
 \begin{align}\label{r2}
  &\ell_{2,\eta}(a,0,x_1,N)\\
\nonumber  &= \hspace{-.4cm}   \sum_{z\in Q_N(x_1), (a_1,a_2)\in \mathfrak{G}_{2,a}}
                \hspace{-.3cm}      \prob{A_{N-M}=a_1,S_{N-M}=z,\tilde \tau > N-M}
                   \   \probz{A_M=a_2,S_M=x_1,\tilde \tau \le M } .
  \end{align}
By using again \eqref{eq:crudeboudsnxtau}, summing over $a_2\leq a$,  applying time reversibility and summing over $z \in Q_N(x_1)$  we get that there exists a $C>0$ 
such that 
 \begin{align}\label{r22}
\nonumber   &\ell_{2,\eta}(a,0,x_1,N)\leq  \frac{C_5}{ (N-M)^{5/2}} \  \PP_{x_1}\big(S'_M\in Q_N(x_1),\tilde \tau' \le M\big)\le  \frac{C}{ N^{5/2}}\   \PP_0\Big(\max_{k\le \eta^3 N} \valabs{S'_k} \ge \eta
      \sqrt{N}\Big),
  \end{align}  
 where we have used the condition $x_1\ge \eta \sqrt{N}$ and the equality $M=\lfloor \eta^3 N\rfloor$ to obtain the last inequality. By Donsker invariance principle 
 we get finally that  that for $N$ large enough
 \be{bfreta2}
 \ell_{2,\eta}(a,0,x_1,N) \leq \frac{C}{N^{5/2}}\  \PP_0\Big(\max_{s\in [0,1] } |B_s|\geq \tfrac{1}{\sqrt \eta}\Big).
 \ee
 At this stage, \eqref{boundArr} combined with \eqref{boundArr2}, \eqref{bfreta1} and \eqref{bfreta2} complete the proof of the first inequality in \eqref{eq:diffe}.
  
We only give the main lines of the proof of the second inequality in \eqref{eq:diffe}, i.e., 
  \begin{align}\label{boundAr}
    & \PP_{x_1}\etp{A_{N}=a, S_{N}= x_2,\tilde \tau > N}  - \alpha_\eta(a,x_1, x_2,N) \\
  \nonumber   &\quad\quad\quad\quad \quad\quad=  \ell_{1,\eta}(a,x_1,x_2,N) -\ell_{2,\eta}(a,x_1,x_2,N)+ \ell_{3,\eta}(a,x_1,x_2,N) -\ell_{4,\eta}(a,x_1,x_2,N),
     \end{align}
 with     
 \begin{align}\label{boundAr2}
 \ell_{1,\eta}(a,x_1,x_2,N)&= \PP_{x_1}\big(A_N=a,S_{N-M}\notin Q_N(x_2),S_N = x_2,\tilde \tau > N\big) \\
\nonumber    \ell_{2,\eta}(a,x_1,x_2,N)&= \PP_{x_1}\big(A_N=a, S_{N-M}\in Q_N(x_2),S_N = x_2, N-M < \tilde \tau \le N\big )  \\
\nonumber   \ell_{3,\eta}(a,x_1,x_2,N)&= \PP_{x_1}\big( A_N=a, S_M\notin Q_N(x_1), S_{N-M}\in Q_N(x_2),S_N = x_2, \tilde \tau>N-M \big )\\
\nonumber   \ell_{4,\eta}(a,x_1,x_2,N)&\\
\nonumber  = \PP_{x_1}\big(&A_N=a, S_M \in Q_N(x_1), \tilde \tau \leq M,  S_{N-M}\in Q_N(x_2),S_N = x_2, S\in W_{M,N-M}\big).
\end{align}
Here as well we  need to identify $\bar G_2(\eta)$ satisfying $\lim_{\eta\to 0} \bar G_2(\eta)=0$ and such that for every $\eta>0$ we have  for $N$ large enough  and for $a\in \N$ and $x_1, x_2\geq \eta \sqrt{N}$ that  $r_{i,\eta}(a,x_1,x_2,N) \leq \bar G_2(\eta)/N^2$ (for  $1\leq i\leq 4$). To that aim we repeat the same type of computations as those who led 
to \eqref{bfreta1} and \eqref{bfreta2}. To be more precise, we apply Markov's property at time $N-M$ for $\ell_{1,\eta}$ and $\ell_{2,\eta}$
and at time $M$ for   $\ell_{3,\eta}$ and $\ell_{4,\eta}$. Since we consider random walk trajectories starting from $x_1$ (instead of starting from the origin in the previous proof) we apply \eqref{eq:simtau2} instead of \eqref{eq:crudeboudsnxtau} and this explains why the r.h.s. in  the second inequality in \eqref{eq:diffe} decays as $N^{-2}$ instead of $N^{-5/2}$. 
To obtain $\bar G_2(\eta)$ we apply Donsker invariance principle on the time interval $\{N-M,\dots,N\}$ for $\ell_{1,\eta}$ and $\ell_{2,\eta}$ and on the time interval $\{0,\dots,M\}$ for 
 $\ell_{3,\eta}$ and $\ell_{4,\eta}$. This completes the proof of the second inequality in \eqref{eq:diffe}.
 
 At this stage we recall that $r_1=\lfloor t_1 N\rfloor$ and we use   \eqref{eq:crudeboudsnxtau} to claim that there exists a $C>0$ such that 
\begin{align}\label{thysec1}
\sum_{a_1=1}^ a  \prob{A_{r_1}=a_1,S_{r_1}=x_1,\tilde \tau > r_1}&\leq \PP_0(S_{r_1}=x_1, \tilde \tau>r_1)\leq \frac{C}{N},
\end{align}
then we use that $r_2=\lfloor t_2 N\rfloor-\lfloor t_1 N\rfloor$ and  \eqref{eq:simtau2} to claim that there exists a $C>0$ such that 
\be{thysec2}
\sum_{a_2=1}^{a}  \PP_{x_1}\etp{A_{r_2}=a_2, S_{r_2}= x_2,\tilde \tau > r_2}  \leq \PP_{x_1}(S_{r_2}=x_2)\leq \frac{C}{\sqrt N},
\ee
and finally we use time reversibility and  \eqref{eq:crudeboudsnxtau} to claim that there exists a $C>0$ such that 
\begin{align}\label{thysec3}
 \sum_{a_3=1}^a \PP_{x_2}\etp{A_{r_3}=a_3,\tilde \tau=r_3,S_{r_3}=0} &\leq \PP_0( \tilde \tau>r_3, S'_{r_3}=x_2)\leq \frac{C}{N}.
 \end{align} 
At this stage, it suffices to combine \eqref{eq:markovdecompositionexcursion} and  \eqref{eq:diffe} with (\ref{thy1}--\ref{thy3}) and (\ref{thysec1}--\ref{thysec3}) to complete the proof of  
\eqref{defG}.

\subsubsection{\bf Proof of (iii)}

We recall \eqref{eq:thm6devatutinwachtel}. We need to establish that there exists $G_2:(0,\infty)\to (0,\infty)$ satisfying
$\lim_{\eta \to 0} G_2(\eta)=0$ and such that for every  $\eta>0$ we have  for $N$ large enough and for  $a\in \N$ and  $0\le x_1 \le
\eta \sqrt{N}$ or $0\le x_2 \le \eta \sqrt{N}$ that
\begin{equation}\label{iaa}
   \cG_{N}(a,x_1,x_2):= \prob{A_N=a,S_{\floor{Nt_1}}=x_1,S_{\floor{Nt_2}}=x_2,\tilde \tau =N,S_N=0} \le \frac{G_2(\eta)}{N^{4}}.
\end{equation}
 We assume that  $0\le x_1\le \eta \sqrt{N}$,  and we use \eqref{eq:tsnxpetit} to bound the second probability in the r.h.s. of  \eqref{eq:markovdecompositionexcursion}. 
 Therefore
  \begin{align}\label{dp3i}
  \nonumber  \cG_{N}(a,x_1,x_2)&\le  \tfrac{\gep_2(\eta)}{ r_2^{2}} \sum_{(a_1,a_3)\in \mathfrak{G}_{2,a}^-} \prob{A_{r_1}=a_1,S_{r_1}=x_1,\tilde \tau > r_1}
  \PP_{x_2}\etp{A_{r_3}=a_3,\tilde \tau=r_3,S_{r_3}=0}
  \\
  &\le \tfrac{\gep_2(\eta)}{r_2^{2}}\  \prob{S_{r_1}=x_1,\tilde \tau > r_1}
\   \PP_{x_2}\etp{\tilde \tau=r_3,S_{r_3}=0}.
  \end{align}
  With the help of  \eqref{eq:crudeboudsnxtau} we bound the first probability in the r.h.s. in \eqref{dp3i} and 
 we use time reversibility and  \eqref{eq:crudeboudsnxtau} again to bound the second probability so that \eqref{dp3i} becomes 
  \begin{align}\label{dp3ii}
   \nonumber  \cG_{N}(a,x_1,x_2)
   & \le \frac{\gep_2(\eta) C_4^2}{ r_2^{2} r_1
 r_3} \le \frac{G_2(\eta)}{ N^{4}},
\end{align}
and \eqref{iaa} is proven.
It remains to consider the case $0\le x_2\le \eta \sqrt{N}$ but by time reversibility 
this is exactly the same proof.
%

\subsubsection{\bf Proof of (iv)}

{\color{black} We recall the definition of $\mathcal{A}(\eta)$ in
  \eqref{defaeta}.} Since $\phi_t$ is Lipschitz continuous (see
\iflong Lemma A.1 \fi) and since $\phi_t(a,x_1,\dots,x_d)=0$ if there exists $i\leq d$ such that  $x_i=0$  we obtain that
$$ \prob{(A_1(e),e_{t_1},\dots,e_{t_d}) \notin \cA(\eta)} \le \sum_{i=1}^{d} \prob{e_{t_i} \le \eta} = o(\eta),$$
so that (iv) of Proposition~\ref{thm:cvloiplusapproxregegallclt} is easily proven.

\subsection{Proof of Proposition~\ref{pro:cvconjointetaurw}}
In this proof again we assume, without loss of generality that $\sigma=1$. We shall prove at the end of this section the following
\begin{lemma}\label{eq:asympataustau}
There exists $C>0$ such that 
  \begin{equation*}
    v_L := \prob{A_{\tilde \tau}=L, S_{\tilde \tau}=0}= C L^{-4/3} (1+o(1)).
    \end{equation*}
    
  \end{lemma}

We let $\etc{a,b}\subset (0,\infty)$ and for $i\leq d$ we let $[\bar a_i,\bar b_i]\subset (0,\infty)$.  We set  $\cR := \prod_{i=1}^d
[\bar a_i,\bar b_i]$ and 
\be{defmats}
\mathcal{S}_L(a,b,\mathcal{R}):=\big\{(N,x)\in \N\times\N^d\colon\, \tfrac{N}{L^{2/3}}\in [a,b],     \tfrac{x}{\sqrt{N}} \in \mathcal{R}\big\},
\ee
and we consider
\begin{align}\label{asym}
  I_L &:=\PP_0\Big(\frac{\tilde \tau }{L^{2/3}} \in \etc{a,b}, \frac{\mathbf S_{\floor{t\tilde \tau}}}{\sqrt{\tilde \tau}} \in \cR \mid A_{\tilde \tau} = L, S_{\tilde \tau} =0\Big) \\
\nonumber      &= \unsur{v_L} \sum_{(N,x)\in \mathcal{S}_L(a,b,\mathcal{R})}
  \prob{A_N=L,\mathbf S_{\floor{tN}}=x\mid \tilde \tau=N, S_{N}=0}  \prob{\tilde \tau=N, S_{N}=0}.
\end{align}
With the help of Proposition \ref{pro:lltbridge}  and then  by using \eqref{eq:thm6devatutinwachtel} and the fact that $\phi_{t}$ is bounded (see Remark \ref{rem1}), {\color{black}we claim that there exists  $\gep_1, \gep_2:\N^{d+2}\to \R$} such that  \eqref{asym} can be rewritten as
 \begin{align} \label{asym2}    
\nonumber I_L &= \unsur{v_L} \sum_{(N,x)\in \mathcal{S}_L(a,b,\mathcal{R})} \frac{1}{N^{r}} \Big[\phi_{t}\Big(\frac{L}{N^{3/2}},\frac{x}{N^{1/2}}\Big) + \gep_1(N,x,L) \Big] \,   \prob{\tilde \tau=N, S_{N}=0}\\
  & = \frac{C}{v_L}  \sum_{(N,x)\in \mathcal{S}_L(a,b,\mathcal{R})} \frac{1}{N^{r+3/2}} \Big[\phi_{t}\Big(\frac{L}{N^{3/2}},\frac{x}{N^{1/2}}\Big) +
  \gep_2(N,x,L)\Big],
  \end{align}
  where for $i\in\{1,2\}$,  $\gep_i(N,x,L)$ converges to $0$ as $N\to \infty$ uniformly in $(x,L)$.
 At this stage, we note that there exists $C>0$ such that 
 $$  \sum_{(N,x)\in\mathcal{S}_L(a,b,\mathcal{R})} \frac{1}{N^{r+3/2}} \le \frac{C}{L^{4/3}},$$
 and therefore it suffices to use  Lemma \ref{eq:asympataustau} to assert that there exists $\tilde C>0$ and  {\color{black}  $\gep_3:\N\to \R$ satisfying 
$\lim_{L\to \infty} \gep_3(L)=0$} such that 
\begin{equation}
  I_L =(1+\gep_3(L))\  \tilde C L^{4/3} \sum_{(N,x)\in\mathcal{S}_L(a,b,\mathcal{R})} \frac{1}{N^{r+3/2}} \ \phi_{t}\Big(\frac{L}{N^{3/2}},\frac{x}{N^{1/2}}\Big).
\end{equation}
We set $\psi(u) := \int_\cR \phi_{t}(u,x)\, dx$  for $u>0$ and we
recall that $\phi_t$ is  Lipschitz continuous (see %
\iflong
Lemma A.1%
\fi
). Thus,  we have by Riemann sum approximation that
there exists {\color{black} $\gep_4:\N^2\mapsto \R$ such that  for $L\in \N$ and  $\frac{N}{L^{2/3}}\in [a,b]$}
$$ \frac{1}{N^{d/2}} \sum_{ x\in \N\colon \atop xN^{-1/2}\in\cR}  \phi_{t}\Big(\frac{L}{N^{3/2}},\frac{x}{N^{1/2}}\Big)= (1+\gep_4(N,L))\ 
\psi\Big(\big(\tfrac{N}{L^{2/3}}\big)^{-3/2}\Big),$$
where $\gep_4(N,L)$ converges to $0$ as $L\to \infty$ uniformly in $\frac{N}{L^{2/3}}
\in [a,b]$. Therefore, recalling that $r=\frac{d}{2}+\frac{3}{2}$, we obtain that there exists  $\gep_5:\N\to \R$ satisfying 
$\lim_{L\to \infty} \gep_5(L)=0$ such that 
$$ I_L= \frac{\tilde C (1+\gep_5(L))}{L^{2/3}}\hspace{-.3cm} \sum_{N\in \N\colon \atop N L^{-2/3}\in\etc{a,b}}
\Big(\frac{N}{L^{2/3}}\Big)^{-3} \psi\Big(\Big(\frac{N}{L^{2/3}}\Big)^{-3/2}\Big) \to_{L\to \infty}  \tilde C \int_{\etc{a,b}\times \cR} 
\frac{1}{u^3} \phi_{t}(\frac{1}{u^{3/2}},x)\, du\, dx.$$
And this establishes the convergence in distribution of the
Proposition, since by Lemma  \ref{lem:norlennorarea}   $ \tilde C \, u^{-3}\, \phi_{t}( u^{-3/2},x)$ is the density of  $\Big(R,\frac{\cE(t_1 R)}{\sqrt{R}},\dots,\frac{\cE(t_d R)}{\sqrt{R}}\Big)$.

 \begin{proof}[of Lemma \ref{eq:asympataustau}]
 We denote by $\bar \phi$ the density of $A_1(e)$. Following the same steps as above, we prove that there exists 
 $C_1>0$  and {\color{black} $\gep_6:\N\to \R$ satisfying $\lim_{L\to \infty} \gep_6(L)=0$ } such that 
   $$ \PP_0\Big(A_{\tilde \tau}=L,S_{\tilde \tau}=0, \frac{\tilde \tau}{ L^{2/3}} \in \etc{a,b}\Big) = \frac{C_1 (1+\gep_6(L))}{L^{4/3}}
   \int_{\etc{a,b}} \frac{1}{u^{3}} \bar{\phi}(\frac{1}{u^{3/2}})\, du.$$

   Then, by applying Markov's property at time $K=\lfloor \frac{N}{2}\rfloor$ combined with the second inequality in \eqref{eq:crudeboudsnxtau}, we  derive easily that
   there exists $C>0$ such that for every $N,L\in \N$ we have  
   $ \prob{A_N=L,S_N=0,\tilde \tau=N} \le C N^{-3}$. As a consequence, there exists $C'>0$ such that for every $L\in \N$, 
    $$\PP_0\Big(A_{\tilde \tau}=L,S_{\tilde \tau}=0, \frac{\tilde \tau}{L^{2/3}}\ge b\Big ) \le C \sum_{N \ge b
     L^{2/3}} \frac{1}{N^{3}} \le \frac{C'}{ (b L^{2/3})^{2}} = \frac{C'}{b^2 L^{4/3}}.$$
 It remains to prove that  there exists a $G:(0,\infty)\to (0,\infty)$ satisfying $\lim_{a \to 0} G(a)=0$ and such that for every $a>0$ and for $L$ large enough, 
  \be{prlqu}
  \PP_0\Big(A_{\tilde \tau}=L,S_{\tilde \tau}=0, \frac{\tilde \tau}{ L^{2/3}}\le a\Big) \le \frac{G(a)}{L^{4/3}}.
   \ee
 To that aim we observe that, for $N\leq a L^{2/3}$, a trajectory $S$ in 
 $\{A_N=L,S_N=0,\tilde \tau=N\}$ necessarily satisfies $\max\{S_i\colon i\leq N\}\geq L/4N$. For this reason we can define 
 \begin{align}
\bar \tau_{L/4N}&:=\min\{i\geq 1\colon\, S_i\geq \tfrac{L}{4N}\}\quad \text{and} \quad\hat \tau_{L/4N}:=\max\{i\leq  N\colon\, S_i\geq \tfrac{L}{4N}\}.
 \end{align}
For $N\leq aL^{2/3}$ we write
$\cO_N:=\{A_N=L,S_N=0,\tilde \tau=N\}= \cB_N\, \cup\, \cC_N\, \cup \cD_N$ with 
\begin{gather}
 \nonumber \cB_N:=\cO_N\cap\big\{ S_{\bar \tau_{L/4N}}\leq \tfrac{3L}{8N}\big\},\quad
\cC_N:=\cO_N\cap\big\{S_{\hat \tau_{L/4N}}\leq \tfrac{3L}{8N}\big\},\\
 \cD_N:=\cO_N\cap \big\{ S_{\bar \tau_{L/4N}}> \tfrac{3L}{8N},  S_{\hat \tau_{L/4N}}> \tfrac{3L}{8N}\big\}.
\end{gather}

We note first that if $S\in \cD_N$ then, before time $N$,  $S$ must have at least two increments larger than $L/8N$ in modulus. These two increments are $X_{\bar \tau_{L/4N}}$ and $X_{\hat \tau_{L/4N}+1}$. As a consequence, we have 
\be{boundD}
 \cup_{N=1}^{aL^{2/3}} \cD_N\subset \big \{\exists i<j\leq aL^{2/3}\colon\, |X_i|\geq \tfrac{L^{1/3}}{8a},\,    |X_j|\geq \tfrac{L^{1/3}}{8a}\big\},
 \ee
and therefore $$\mathbb{P}_0\big(\cup_{N=1}^{aL^{2/3}} \cD_N\big)\leq (aL^{2/3})^2\  \mathbb{P}\big(|X_1|\geq \tfrac{L^{1/3}}{8a}\big)^2.$$ 
Since 
$X_1$ has finite fourth moment, Markov inequality yields $\displaystyle \mathbb{P}\big( |X_1|\geq \tfrac{L^{1/3}}{8a} \big)\leq \mathbb{E}(X_1^4)  \tfrac{a^{4}}{L^{4/3}}$
and  therefore $\displaystyle \mathbb{P}\big(\cup_{N=1}^{aL^{2/3}} \cD_N\big)\leq C \tfrac{a^{10}}{L^{4/3}}$ for some $C>0$. 

\smallskip

 At this stage we note that it is sufficient to focus on $\sum_{N=1}^{aL^{2/3}} \mathbb{P}_0(\cB_N)$ since by time reversal 
 $\sum_{N=1}^{aL^{2/3}} \mathbb{P}_0(C_N)$ is taken care of similarly. For $N\leq aL^{2/3}$ we decompose the set  $\cB_N$ 
 depending on $(r_1, r_2, r_3)$ the value taken by $(\bar \tau_{L/4N}, \hat \tau_{L/4N}-\bar \tau_{L/4N},N-\hat \tau_{L/4N})$, on  
 $(\alpha_1,\alpha_2,\alpha_3)$ the value taken by $(A_{r_1-1}, A_{r_2}-A_{r_1-1}, L-A_{r_2})$ and on 
 $(x_1,x_2)$ the values taken by $(S_{\bar \tau_{L/4N}}, S_{\hat \tau_{L/4N}})$. We note that, by construction, $\alpha_1$ and 
 $\alpha_3$ are necessarily smaller than $L/4$. We apply Markov's property at times $r_1$ and $r_1+r_2$ and reversibility for the last piece of trajectory 
(between times  $r_1+r_2$ and $N$) so that (recall \eqref{defGa})
 \be{masteq}
\mathbb{P}_0(\cB_N) =\sum_{(r_1,r_2,r_3)\in \mathfrak{G}_{3,N}}\sum_{\genfrac{}{}{0pt}{}{(\alpha_1, \alpha_2,\alpha_3)\in \mathfrak{G}_{3,L}}{\colon \alpha_1, \alpha_3\leq L/4 }}\sum_{x_1=\frac{L}{4N}  }^{\frac{3L}{8N}}    \, \sum_{ x_2=\frac{L}{4N}}^\infty
T_{r_1,\alpha_1,x_1}\,  W_{x_1,x_2,r_2,\alpha_2}\,   Y_{r_3,\alpha_3,x_2},
 \ee
 with
\begin{align}
\nonumber T_{r_1,\alpha_1,x_1}&:=\prob{r_1=\bar \tau_{L/4N}<\tilde \tau,\,  S_{r_1}=x_1,\,  A_{r_1-1}=\alpha_1} \\
W_{x_1,x_2,r_2,\alpha_2}&:=\mathbb{P}_{x_1}\big(\tilde \tau>r_2,\,  x_1+ A_{r_2}=\alpha_2,\,  S_{r_1}=x_2 \big)\\
\nonumber Y_{r_3,\alpha_3,x_2}&:= \prob {r_3=\bar \tau'_{L/4N}<\tilde \tau',\,  S'_{r_3}=x_2,\,   A_{r_3-1}=\alpha_3 }.
\end{align}
Our goal now, is to find an upper bound on $W_{x_1,x_2,r_2,\alpha_2}$ which only depends on $N$ and $L$. 
To that aim, we  remove the constraint $\tilde \tau>r_2$ to obtain the upper bound
\be{upbW}
W_{x_1,x_2,r_2,\alpha_2}\leq \prob{S_{r_2}=x_2-x_1, \,  A_{r_2}=\alpha_2-(1+r_2)\,  x_1}.
\ee
We recall below the LCLT established in \cite[Proposition 4.7]{CP15}, which is an improved version of  \cite[Proposition 2.3]{CD08}. 
\begin{proposition}\label{lltalgarea}
Let  $(X_i)_{i\geq 1}$ be an IID sequence of integer-valued centered random variables of variance $1$ and satisfying \eqref{defhyprw}. Then,
\be{ala}
\sup_{n\in \N}  \, \sup_{k,a\in \Z} n^3\,  \Big| \PP_0(S_n=k, A_n=a)-\tfrac{1}{\sigma^2\,  n^2}\, g\Big(\tfrac{k}{\sigma_\beta \sqrt{n}}, \tfrac{a}{\sigma_\beta n^{3/2}}\Big)\Big|< \infty,
\ee
with $g(y,z)=\frac{6}{\pi} e^{-2 y^2-6 z^2 +6 yz}$ for $(y,z)\in \R^2$.
\end{proposition}
We observe that there exist $c_1,c_2>0$ such that 
$g(y,z)\leq c_1 e^{-c_2 z^2}$ for every $(y,z)\in \R^2$, and also that 
the area $\alpha_2-(1+r_2) x_1\geq \frac{L}{8}$ because $x_1\leq \frac{3L}{8N}$, $1+r_2\leq N$ and $\alpha_2\geq \frac{L}{2}$.   We begin with the case  $r_2 \geq N^{2/3}$ and we apply Proposition \ref{lltalgarea} to obtain that there exist
$C>0$ and $c_3>0$ such that 
\begin{align}
\nonumber W_{x_1,x_2,r_2,\alpha_2}&\leq \frac{C}{r_2^3}+\frac{c_1}{r_2^2} \exp\Big[-c_2 \frac{(\alpha_2-(1+r_2) x_1)^2}{r_2^{3}}\Big] \leq \frac{C}{r_2^3}+\frac{c_1}{r_2^2} \exp\Big[-c_3 \Big(\frac{L^2}{r_2^{3}}\Big)\Big]\\
&\leq \frac{C}{r_2^3}+\frac{c_1}{L^{4/3}} 
\Big(\frac{L^{2/3}}{r_2}\Big)^2 \exp\Big[-c_3 \big(\frac{L^{2/3}}{r_2}\big)^3\Big].
\end{align}
Since $x\mapsto x^2 e^{-c_3 x^3}$ is bounded  on $[0,\infty)$, we conclude  that provided $C$ is chosen large enough we have for $L\in \N, \,  N\leq aL^{2/3}$ and  $r_2\geq N^{2/3}$
\be{boundW}
W_{x_1,x_2,r_2,\alpha_2}\leq \frac{C}{N^2}+\frac{C}{L^{4/3}}.
\ee
For $r_2\leq N^{2/3}$, we recall that $\alpha_2-(1+r_2) \, x_1\geq \frac{L}{8}$ and therefore  
$A_{r_2}=\alpha_2-(1+r_2)\,  x_1$ yields that $S^*_{r_2}:=\max\{|S_i|, i\leq r_2\}\geq \frac{L}{8 r_2}$. Consequently, by using the fact that $(S_i)_{i\in \N}$ is a martingale with finite fourth moment we apply Markov's property and Doob's inequality to obtain that there exist  $C, C'>0$  such that  for  $L\in \N$,   $N\leq aL^{2/3}$ and  $r_2\leq N^{2/3}$
\be{ineqmar}
W_{x_1,x_2,r_2,\alpha_2}\leq \prob{S^*_{r_2}\geq \tfrac{L}{8 r_2}}\leq \   C\,  \mathbb{E}_0\Big[\big(S_{r_2}\big)^4\Big] \, \frac{r_2^4}{L^4}\leq\  C' \frac{r_2^6}{L^4}\leq\  C' \frac{N^4}{L^4}.
\ee
Going back to \eqref{masteq}, we use  (\ref{boundW}-\ref{ineqmar}) to assert that provided 
$C$ is chosen large enough we have 
\begin{align}\label{cling}
\mathbb{P}_0(\cB_N) &\leq C \sum_{(r_1,r_2,r_3)\in \mathfrak{G}_{3,N}}\sum_{\genfrac{}{}{0pt}{}{(\alpha_1, \alpha_2,\alpha_3)\in \mathfrak{G}_{3,L}}{\colon \alpha_1,\alpha_3\leq L/4 }}\sum_{x_1=\frac{L}{4N}  }^{\frac{3L}{8N}}    \, \sum_{ x_2=\frac{L}{4N}}^\infty
T_{r_1,\alpha_1,x_1}\, \Big(\frac{1}{N^2}+\frac{1}{L^{4/3}}+\frac{N^4}{L^4}\Big) \,   Y_{r_3,\alpha_3,x_2},
\end{align}
and we observe also that there exists a $C>0$ such that for every $L\in \N$ and $N\leq aL^{2/3}$, 
\be{sommeta}
\sum_{r_1=1}^{N} \sum_{\alpha_1=1}^{\frac{L}{4}} \sum_{x_1=\frac{L}{4N}}^{\frac{3L}{8N}} 
T_{r_1,\alpha_1,x_1}\leq \prob{\bar \tau_{\frac{L}{4N}}<\tilde \tau}\leq C\frac{N}{L}\quad \text{and} \quad 
\sum_{r_3=1}^{N} \sum_{\alpha_3=1}^{\frac{L}{4}} \sum_{x_3=\frac{L}{4N}}^{\infty} 
Y_{r_3,\alpha_3,x_2}\leq C\frac{N}{L}.
\ee
Combining \eqref{cling} with (\ref{sommeta}) we conclude that provided  $C>0$ is chosen large enough, 
we have for every $L\in \N$ and $N\leq a L^{2/3}$ that $\prob{\cB_N}\leq
C
\big(\frac{1}{L^2}+\frac{N^2}{L^{2+\frac{4}{3}}}+\frac{N^6}{L^6}\big)$
and therefore 
\be{finbouA}
\sum_{N=1}^{aL^{2/3}} \prob{\cB_N}\leq C \Big(\frac{a}{L^{4/3}}+\frac{a^3}{L^{4/3}}+\frac{a^7}{L^{4/3}}\Big),
\ee
which completes the proof of \eqref{prlqu}.

 \end{proof}

\subsection{Tightness}\label{subsec:tightness}
\begin{lemma}\label{lem:tight}
  Let $\nu_L$ be the  distribution of the process 
\begin{equation}
\etp{\frac{\tilde \tau}{L^{2/3}} ;
  \unsur{\sqrt{\tilde \tau}}\valabs{S_{\floor{s\tilde \tau}}}, s\in [0,1] },
\end{equation}
with $S$ sampled under 
  $\prob{ . \mid A_{\tilde \tau}=L, S_{\tilde \tau}=0}$. Then $(\nu_L,
L\ge 1)$ is a tight sequence.
\end{lemma}

\begin{proof}
We consider  the continuity modulus of a given function
$f:[0,1]\mapsto \R$ : given $x<y\in [0,1]$ and $\gep>0$ we set 
\be{modcon}
\Gamma_{[x,y]}(f,\gep):= \sup \big\{|f(t)-f(s)|\colon s,t \in [x,y]; t-s\leq \gep\big\}.
\ee
In what follows, we denote by $\hat S_{\tilde \tau}$ the process $\big\{ \frac{1}{\sqrt{\tilde \tau}} S_{\lfloor t\tilde \tau\rfloor}, t\in [0,1]\big\}$ and for every $N\in \N$ 
by $\tilde S_N$ the process $\big\{ \frac{1}{\sqrt{N}} S_{\lfloor t N\rfloor}, t\in [0,1]\big\}$.  By reversibility, Lemma
\ref{lem:tight} will be proven once we show that for every $\delta>0$,
\be{prut}
\lim_{\gep \to 0} \limsup_{L\to \infty} \prob{\Gamma_{[0,\frac{1}{2}]}\big(\hat S_{\tilde \tau},\gep\big)>\delta \, \mid\,  A_{\tilde \tau}=L,\, S_{\tilde \tau}=0}=0.
\ee
A slight modification of Proposition~\ref{pro:cvconjointetaurw} (whose proof is completely similar) guarantees us that, when sampled under   $\prob{. \mid A_{\tilde \tau} =L, S_{\tilde \tau}=0}$ the 
random vector 
$$\displaystyle \Big(\frac{\tilde\tau}{ L^{2/3}},\frac{A_{\tilde \tau/2}}{L},  \frac{S_{\tilde \tau/2}}{L^{1/3}}\Big)\quad \text{
converges in distribution to}
\quad \displaystyle \Big( R,\, \frac{1}{R^{3/2}} \int_0^{R/2} \cE(u)\, du,\,  \unsur{\sqrt{R}}\cE\Big(\frac{R}{2}\Big)\Big).$$
Hence 
for every $\gep>0$ there exist $[c_1,c_2],  [y_1,y_2] \subset (0,\infty)$ and $[\alpha_1,\alpha_2]\subset (0,1)$, such that 
\be{eq:liminfcont}
\liminf_{L\to \infty}\,  \prob{\cK_{L,c,\alpha,y}\, \mid \,  A_{\tilde \tau}=L,\, S_{\tilde \tau}=0}\geq 1-\gep,
\ee
with 
\be{defAm}
\cK_{L,c,\alpha,y}:=\bigg\{\frac{\tilde \tau}{L^{2/3}}\in [c_1,c_2], \frac{A_{\tilde \tau/2}}{L}\in [\alpha_1,\alpha_2],\frac{S_{\tilde \tau/2}}{L^{1/3}}\in [y_1,y_2]\bigg\}.
\ee
Note that we can safely assume that
$\frac{y_1}{\sqrt{c_1}}<\frac{y_2}{\sqrt{c_2}}$ and thus 
\be{prut1}
 \prob{\Gamma_{[0,\frac{1}{2}]}\big(\hat S_{\tilde \tau},\gep\big)>\delta \, \mid\,  A_{\tilde \tau}=L,\, S_{\tilde \tau}=0}\leq B_{1,L}+B_{2,L},
 \ee
 with
 \begin{align}
 B_{1,L}&=
  \prob{\cK_{L,c,\alpha,y}^c\, \mid \,  A_{\tilde \tau}=L,\, S_{\tilde \tau}=0},\\
\nonumber  B_{2,L}&=  \bbP_0\big(\Gamma_{[0,\frac{1}{2}]}\big(\hat S_{\tilde \tau},\gep\big)>\delta \, \mid\, V\in \cK_{L,c,\alpha,r},\,   A_{\tilde \tau}=L,\, S_{\tilde \tau}=0\big).
 \end{align}
 The quantity $B_{1,L}$ is taken care of with \eqref{eq:liminfcont}, so that we only need to focus on $B_{2,L}$. To that aim, we partition 
 the event $\cK_{L,c,\alpha,y}$ depending on the value $N$ taken by $\tilde \tau$, i.e.,  
 \be{defBB}
 B_{2,L}=\frac{\sum_{N=c_1 L^{2/3}}^{c_2 L^{2/3}}
   \Lambda_{1,L,N}}{\sum_{N=c_1 L^{2/3}}^{c_2 L^{2/3}} \Lambda_{2,L,N}},
 \ee

 \begin{align*}
 \text{with}\quad \Lambda_{1,L,N}&=\sum_{a=\alpha_1 L}^{\alpha_2 L}\,  \sum_{v=y_1 L^{1/3}}^{y_2 L^{1/3}} b_{N,v,a} \, {\widehat b}_{N,v,L-a}\quad \text{and}\quad   \Lambda_{2,L,N}=\sum_{a=\alpha_1 L}^{\alpha_2 L} \, \sum_{v=y_1 L^{1/3}}^{y_2 L^{1/3}} {\widehat b}_{N,v,a}\,  {\widehat b}_{N,v,L-a}, 
  \end{align*}
\begin{align*}
 \text{and with}\qquad b_{N,v,a}&=\prob{\Gamma_{[0,\frac{1}{2}]}\big(\tilde S_N,\gep\big)>\delta,\,  S_{\frac{N}{2}}=v,\,  A_{\frac{N}{2}}=a,\,  \tilde \tau>\tfrac{N}{2}},\\
{\widehat b}_{N,v,k}&=\prob{S_{\frac{N}{2}}=v,\,  A_{\frac{N}{2}}=k,\,  \tilde \tau>\tfrac{N}{2}},
\end{align*}
where we have used Markov's property at $N/2$ and time reversal in $\widehat b$.
Let us finally set for $N\in L^{2/3} [c_1,c_2]$, 
 \be{eq:upb}
 K_{1,L,N}=\sum_{a=\alpha_1 L}^{\alpha_2 L}\,  \sum_{v=y_1 L^{1/3}}^{y_2 L^{1/3}} b_{N,v,a}\leq \prob{\Gamma_{[0,\frac{1}{2}]}\big(\tilde S_N,\gep\big)>\delta, \tilde \tau>\tfrac{N}{2}},
 \ee
 and 
\begin{align}\label{eq:bud}
 K_{2,L,N}=\sum_{a=\alpha_1 L}^{\alpha_2 L} \, \sum_{v=y_1 L^{1/3}}^{y_2 L^{1/3}} {\widehat b}_{N,v,a}&= \prob{S_{\frac{N}{2}}\in L^{1/3}\,[y_1,y_2], 
 A_{\frac{N}{2}} \in L\, [\alpha_1,\alpha_2] ,  \tilde \tau>\tfrac{N}{2}}\\
\nonumber  &\geq \PP_0\Big(S_{\frac{N}{2}}\in \sqrt{N}\,\Big[\tfrac{y_1}{\sqrt{c_1}},\tfrac{y_2}{\sqrt{c_2}}\Big], 
 A_{\frac{N}{2}} \in  N^{3/2}\,\Big[\tfrac{\alpha_1}{c_1^{3/2}},\tfrac{\alpha_2}{c_2^{3/2}}\Big] ,  \tilde \tau>\tfrac{N}{2}\Big),
\end{align}
where we have used that $c_1 L^{2/3}\leq N\leq c_2 L^{2/3}$ in combination with the two inequalities \\ $\frac{y_1}{\sqrt{c_1}}<\frac{y_2}{\sqrt{c_2}}$ and $\frac{\alpha_1}{c_1^{3/2}}<\frac{\alpha_2}{c_2^{3/2}}$. Thus, with (\ref{eq:upb}--\ref{eq:bud}) we can safely write 
\be{upp}
\frac{K_{2,L,N}}{K_{1,L,N}}\leq \frac{\prob{\Gamma_{[0,\frac{1}{2}]}\big(\tilde S_N,\gep\big)>\delta\, \mid\, \tilde \tau>\frac{N}{2}}}{\PP_0\Big(S_{\frac{N}{2}}\in \sqrt{N}\,\Big[\tfrac{y_1}{\sqrt{c_1}},\tfrac{y_2}{\sqrt{c_2}}\Big], 
 A_{\frac{N}{2}} \in  N^{3/2}\,\Big[\tfrac{\alpha_1}{c_1^{3/2}},\tfrac{\alpha_2}{c_2^{3/2}}\Big]\, \mid\,  \tilde \tau>\tfrac{N}{2}\Big)}.
\ee
At this stage, we use the equality $\tilde S_N(s)=\frac{1}{\sqrt{2}} \tilde S_{\frac{N}{2}}(2s)$ for every $s\in [0,1/2]$  and the convergence in distribution proven in \cite{B76}, i.e.,   
\be{convmean}
\lim_{N\to \infty} \Big( \tilde S_{\frac{N}{2}}\, ;\, \PP_0\Big(\cdot\,  \mid \, \tilde \tau>\tfrac{N}{2}\Big)\Big)=( \cM_s)_{s\in [0,1]},
\ee
where $\cM$ is the standard Brownian meander (see \cite{Imhof84}), to assert that 
\be{eq:limsumea}
\lim_{\gep\to 0} \limsup_{N\to \infty} \prob{\Gamma_{[0,\frac{1}{2}]}\big(\tilde S_N,\gep\big)>\delta\, \big | \, \tilde \tau>\tfrac{N}{2}}=0,
\ee
and that 
\begin{multline}\label{eq:limsumar}
\lim_{N\to \infty } \PP_0\Big(S_{\frac{N}{2}}\in \sqrt{N}\,\Big[\tfrac{y_1}{\sqrt{c_1}},\tfrac{y_2}{\sqrt{c_2}}\Big], 
 A_{\frac{N}{2}} \in  N^{3/2}\,\Big[\tfrac{\alpha_1}{c_1^{3/2}},\tfrac{\alpha_2}{c_2^{3/2}}\Big]\,\big |\,  \tilde \tau>\tfrac{N}{2}\Big)\\
 =\PP\Big( \cM^+_1 \in \Big[\tfrac{\sqrt{2}y_1}{\sqrt{c_1}},\tfrac{\sqrt{2} y_2}{\sqrt{c_2}}\Big], 
 A_{1}(\cM^+) \in  \Big[\tfrac{2^{3/2}\alpha_1}{c_1^{3/2}},\tfrac{2^{3/2}\alpha_2}{c_2^{3/2}}\Big]\Big),
\end{multline}
with $A_1(\cM^+)=\int_0^{1} \cM^+_s ds$. The r.h.s. in \eqref{eq:limsumar} being positive, we deduce from (\ref{eq:limsumea}--\ref{eq:limsumar}) that
\begin{align}\label{eq:equalims}
\lim_{\gep\to 0} \limsup_{L\to \infty} \sup_{N\in L^{2/3} [c_1,c_2]} \frac{K_{1,L,N}}{K_{2,L,N}}= \lim_{\gep\to 0} \limsup_{N\to \infty} \sup_{L\in N^{3/2} [\frac{1}{c_2},\frac{1}{c_1}]} \frac{K_{1,L,N}}{K_{2,L,N}}=0,
\end{align}
where the first equality in \eqref{eq:equalims} is just an inversion of the supremums over $N$ and $L$.  Coming back to 
\eqref{defBB} we can rewrite 
\be{dflm}
 B_{2,L}=\frac{\sum_{N=c_1L^{2/3}}^{c_2 L^{2/3}} \Lambda_{2,L,N} \frac{\Lambda_{1,L,N}}{\Lambda_{2,L,N}}}{\sum_{N=c_1L^{2/3}}^{c_2 L^{2/3}} \Lambda_{2,L,N}},
 \ee
 so that the proof of the step will be complete once we show that 
 \begin{align}\label{eq:limsuarr}
\lim_{\gep\to 0} \limsup_{L\to \infty} \sup_{N\in L^{2/3} [c_1,c_2]} \frac{\Lambda_{1,L,N}}{\Lambda_{2,L,N}}=0.
\end{align}
It remains to use a  Proposition~\ref{pro:llttausupncomplet} to check that there exist $C_1<C_2$ such that for every $N\geq 1$
\begin{align}\label{pru}
C_1 \max_{v\in \sqrt{N} [x_1,x_2], a\in \N^{3/2} [y_1,y_2]} {\widehat b}_{N,v,a}&\leq  \min_{v\in \sqrt{N} [x_1,x_2], a\in \N^{3/2} [y_1,y_2]} {\widehat b}_{N,v,a}\\
\nonumber &\leq C_2 \max_{v\in \sqrt{N} [x_1,x_2], a\in \N^{3/2} [y_1,y_2]} {\widehat b}_{N,v,a},
\end{align}
which is sufficient to assert that \eqref{eq:equalims} implies \eqref{eq:limsuarr}.
\end{proof}

\subsection{Time changing. End of the proof of Theorem \ref{thm:a}}\label{sub:timechange}

For notational convenience, we set 
$\nu(\cE):=(\nu(\cE)(s); s\in [0,1])$ with $\nu (\cE)(s)=\tfrac{1}{\sqrt{R_\cE}}\cE_{s R_\cE}$ for $s\in [0,1]$
 and also  $\hat{S}:=(\hat{S}(s); s\in [0,1])$ with $\hat{S}(s) = \unsur{\sqrt{\tilde \tau}} S_{\floor{s\tilde \tau}}$ for
$s\in [0,1]$. We will denote by $D_{[0,1]}$ the set of cadlag functions defined on $[0,1]$ endowed with the Skorokhod topology
and by $C_{[0,1]}$ the set of continuous function on $[0,1]$ endowed with the uniform metric. From the combination of tightness and finite dimensional
convergence, we deduce that as $L\to \infty$ and with $S$ sampled from $\prob{. \mid A_{\tilde \tau}= L, S_{\tilde \tau}=0}$ 
\begin{equation}\label{convcentr}
 \Big(\frac{\tilde \tau}{L^{2/3}}, \hat{S}\Big) \ \text{converges in distribution to}\  \big(R_\cE, \nu(\cE)\big).
\end{equation}
Then, we introduce the operator $\hat A:D_{[0,1]}\to C_{[0,1]}$  defined for $V\in D_{[0,1]}$ and  $t\in [0,1]$ as 
$$ \hat A(V)(t) = \int_0^t V(s)\, ds.$$ 
It is not difficult to check that any $V\in C_{[0,1]}$ is a  point of continuity for $\hat A$. Thus, since $\nu(\cE)$ is continuous, we deduce from 
\eqref{convcentr} that as $L\to \infty$ and with $S$ sampled from $\prob{. \mid A_{\tilde \tau}= L, S_{\tilde \tau}=0}$
\be{eq:cvavecaire}
\Big( \frac{\tilde \tau}{L^{2/3}},\,  \hat{S},\,  \hat{A}(\hat S),\,  \hat A^{-1}\Big) \ \text{converges in distribution to} \ 
\big(R_\cE,\, \nu(\cE),\,  \hat A(\nu(\cE)),\,  a_{\nu_{\cE}}\big),
\ee
where the fourth coordinate in both sides of \eqref{eq:cvavecaire}  corresponds to the right-continuous pseudo-inverse of the third coordinate, i.e.,  for $u>0$, 
\be{gh}
 \hat{A}^{(-1)}(u):= \inf\big\{t>0 : 
   \hat{A}(\hat S)(t) \ge u\big\},
\ee
and $a_{\nu_\cE}$ is the pseudo-inverse of $\hat A(\nu(\cE))$ defined as in \eqref{defae}.

\medskip

At this stage, by  recalling \eqref{defAA} and \eqref{defae}, we note that  for every $t\in [0,1]$ we have  $\hat A(\hat S)(t)=\frac{1}{\tilde \tau^{3/2}} A_{\floor{t\tilde \tau}}$ and $\hat A(\nu(\cE))(t)=\tfrac{1}{R_\cE^{3/2}} \, A_{tR_{\cE}}(\cE)$.
We recall also \eqref{defchi} and  with $\mathfrak{R}_L:= \frac{\tilde \tau}{L^{2/3}}$ 
we have 
\begin{equation}
  \unsur{\tilde \tau} \, \chi_{sL} = \inf\ens{t > 0: 
    A_{\floor{t\tilde \tau}} \ge sL} = \inf\Big\{t>0 : 
    \hat{A}(\hat S)(t) \ge \tfrac{s}{ \mathfrak{R}_L^{3/2} }\Big\}=\hat{A}^{(-1)}\Big(\tfrac{s}{\mathfrak{R}_L^{3/2}}\Big ),
\end{equation}
so that for $s\in [0,1]$, 
\be{idcst}
\frac{1}{L^{1/3}}\, S_{\chi_{sL}} = \sqrt{\frac{\tilde \tau}{L^{2/3}}} \ \hat{S}\Big(\hat{A}^{(-1)}\Big(\tfrac{s}{\mathfrak{R}_L^{3/2}}\Big)\Big).
\ee
Since the limiting processes in \eqref{eq:cvavecaire} are
continuous,  we  compose  $\hat{S}$ with the
inverse process $\hat{A}^{-1}$ and use \cite[Lemma 2.3]{Kurtz91}  to obtain that as $L\to \infty$ and with $S$ sampled from $\prob{. \mid A_{\tilde \tau}= L, S_{\tilde \tau}=0}$ 
$$\Big(\frac{1}{L^{1/3}}\,  S_{\chi_{sL}};\, 0\le s\le 1\Big)
\quad \text{converges in distribution to}\quad 
\Big( \nu(\cE)\Big(a_{\nu(\cE)}\big(\tfrac{s}{ R_\cE^{3/2}}\big)\Big);\, 0\leq s\leq 1\Big)$$
for the Skorokhod distance. We conclude by observing that  $\nu(\cE)(a_{\nu_{\cE}}(s
R_\cE^{-3/2}))= \cE_{a_{\cE}(s)}$ for $s\in [0,1]$.

\section{Proof of Theorem \ref{thm:ab}}\label{excnorm}
We will not display here all the details of the proof of Theorem \ref{thm:ab} since it is very close in spirit to that of Theorem \ref{thm:a}. 
We will rather insist on the differences between both proofs. 

We recall \eqref{defbe}
and, given
$t=(t_1,\ldots,t_d)$ with $0< t_1< \cdots <t_d\le 1$, we denote by 
$k(y_1,\ldots,y_d)$ the density of the random vector
$\mathbf B_t=(B_{t_1},\ldots,B_{t_d})$, where $(B_s)_{s\in [0,\infty)}$ is a standard Brownian motion.

The key point consists in proving the counterpart of Proposition \ref{pro:lltbridge} in the present framework, i.e., 
 \begin{proposition}\label{pro:llttausupncomplet} With $r= d + \frac32$  and $[h_1,h_2]\subset (0,\infty)$
\begin{align}\label{propexc} 
\lim_{N\to +\infty}  \sup_{a\in [h_1,h_2] N^{3/2}}  \sup_{(x,y)\in \N^d\times  \Z^{d}} \Big |
N^r \mathbb{P}_0(A_N=a,\, 
  &\snt=x, \bar{\mathbf S}_{\floor{Nt}}=y \mid \tilde \tau = N,\, S_N=0) \\
 \nonumber & - \frac{1}{\sigma^{r}} \phi_t\etp{\frac{a}{\sigma
    N^{3/2}}, \frac{x}{\sigma N^{1/2}}} k\etp{\frac{y}{\sigma N^{1/2}}}
\Big |=0.
\end{align}
\end{proposition}
\begin{remark}
Note that contrary to what we wrote  in Proposition \ref{pro:lltbridge}, the supremum in \eqref{propexc} is restricted to $a\in [h_1,h_2] N^{3/2}$. Thanks to Lemma \ref{eq:asympataustau}, this  is sufficient to complete the proof of Theorem \ref{thm:ab}. It would have been  sufficient to take the same restriction in Proposition \ref{pro:lltbridge} to prove Theorem  \ref{thm:a}.
\end{remark}
Proving that with $S$ sampled under $\PP_0(\cdot |  \tilde \tau = N,\, S_N=0)$ and as $N\to \infty$,  the random vector $\bar {\mathbf S}_{\lfloor t N \rfloor}/\sqrt{N}$ converges in distribution towards 
$\mathbf B_t$ is difficult. For this reason we can not apply Proposition \ref{thm:cvloiplusapproxregegallclt} directly to prove \eqref{propexc}. To overstep this difficulty, we first state a 
variant of Proposition \ref{propexc} involving the Brownian meander instead of the excursion. To that purpose, for $v\in [0,1]$, we let 
$\cM^{v,+}=(\cM^{v,+}_s, 0\le s\le v)$ be the Brownian meander on $[0,v]$ (see \cite{Imhof84} for a definition of $\cM^{v,+}$ and note that we will omit the $v$ dependency when $v=1$) . For
$t=(t_1,\ldots,t_d)$ with $0<t_1<\dots<t_d\leq v$ we denote by  $\phi_{t}^{v,+}(a,x),a\in \R,x\in\R^d$ the density of 
$$ \Big(A_v(\cM^{v,+})= \int_0^v \cM^{v,+}(s)\, ds,\, 
 {\boldsymbol \cM}^{v,+}_t:=\big(\cM^{v,+}_{t_1},\ldots,\cM^{v,+}_{t_d}\big)\Big).$$
  
\begin{proposition}\label{pro:llttausupncomplett} 
With $r= d + \frac32$
\begin{align}\label{propexc2}
\lim_{N\to +\infty} \sup_{(a,x,y)\in  \N^{d+1}\times  \Z^{d}} \Big |N^r\,  \PP_0(A_N=a,\, 
 & \snt=x,\,  \bar{\mathbf S}_{\floor{Nt}}=y \mid \tilde \tau > N) \\
\nonumber &- \frac{1}{\sigma^{r}} \phi_{t}^+\etp{\frac{a}{\sigma
    N^{3/2}}, \frac{x}{\sigma N^{1/2}}}  k\etp{\frac{y}{\sigma N^{1/2}}}\Big |=0.
\end{align}
\end{proposition}
The  proof of Proposition \ref{pro:llttausupncomplet} is displayed in Section \ref{secconvlaw} below.  With Proposition 
\ref{pro:llttausupncomplet} in hand, we can prove \eqref{propexc}. To that aim, we apply Markov's property at time $t_d N$ 
in the probability in \eqref{propexc}. Then, we use \eqref{propexc2} to estimate the probability associated to the time window 
$[0, t_d N]$ and a simplified version of \eqref{propexc2} (without $\bar S$ and with $d=1$) to estimate the probability associated to the time window 
$[t_d N, N]$. Then, it remains to perform a Riemann sum approximation and to note that there exists a $C>0$ such that 
for $a\geq 0$ and $x=(x_1,\dots, x_d)\in (0,\infty)^d$ we have 
\be{acr}
\phi_t(a,x)=\frac{C}{\sqrt{t_d (1-t_d)}}\int_0^a \phi_t^{t_d,+}(u,x)\,  \phi^{1-t_d,+}_{1-t_d}(a-u,x_d) \, du
\ee
to complete the proof of Proposition \ref{propexc}. Note that \eqref{acr} is obtained by using classical absolute continuity relationship between 
Brownian excursion, Bessel bridges of dimension $3$ and Brownian meander.
 At this stage,  the rest of the proof of Theorem \ref{thm:ab} is completely similar to that of Theorem \ref{thm:a} and therefore we do not repeat it here.
%
%

\subsection{Proof of Proposition~\ref{pro:llttausupncomplet}}\label{secconvlaw}
To prove the proposition we apply Proposition \ref{thm:cvloiplusapproxregegallclt} and follow
\emph{mutatis mutandis} the path taken to prove
Proposition~\ref{pro:lltbridge}. As mentioned above, the only additional difficulty consists
in proving the convergence in distribution required to apply Proposition \ref{thm:cvloiplusapproxregegallclt}. In this case it means  proving that 
when $S$ is sampled from $\prob{. \mid \tilde  \tau > N}$ and as $N\to \infty$, the random vector $(A_N/N^{3/2},\snt/\sqrt{N},
\bar{\mathbf S}_{\floor{tN}}/\sqrt{N})$  converges in distribution
to $(A_1(\cM^+),\boldsymbol \cM^+_t,\mathbf B_t)$. This will be the object of the rest of the present section.  

%

We shall follow closely the proof of \cite{B76} and therefore stick to
the notations of this paper.
Therefore, we let $Y_n(s)$ be the continuous process on
$[0,+\infty)$ for which $Y_n(\frac{k}{n})= \frac{S_k}{\sigma
  n^\undemi}$ and which is linearly interpolated elsewhere:
\begin{equation*}
  Y_n(s) = \unsur{\sigma
  n^\undemi} \etp{S_{\floor{ns}} + (ns -\floor{ns})\,  X_{1+\floor{ns}}},
  \quad  \quad (s\ge 0).
\end{equation*}
We recall \eqref{defStiS2} and similarly  we let $\bar{Y}_n(s)$ be the linear interpolation
process associated to $(\bar{S}_k)_{k\ge 0}$.
Donsker's Theorem states that $(Y_n(s), \bar Y_N(s),0\le t\le 1)$ converges in
distribution on $(C[0,1]\times C[0,1],\rho)$ to two independent  Brownian motions, where $C[0,1]$ is
the space of continuous functions and $\rho$ the uniform metric.
\smallskip

At this stage, we set $\cC^+=\ens{f \in C[0,1]: f(s)\ge 0, \forall s\in\etc{0,1}}$ and the  convergence in law that we are looking for will be a straightforward consequence of the following proposition.
\begin{proposition}\label{pro:lctlboltdecorr}
  Let $B$ be a standard Brownian motion independent of the
  Brownian meander $\cM^+$. {\color{black}Then, as $n\to \infty$, with $S$ sampled under  $\bbP_0(\cdot \, |\,  Y_n \in
        \cC^+)$ the}
process
  \begin{equation*}
    \etp{Y_n(s),\bar{Y}_n(s); 0\le s\le 1}\ \text{converges in distribution to}\  (\cM^+_s, B_s; 0\le s\le 1).
  \end{equation*}
\end{proposition}

  \begin{proof}[of Proposition \ref{pro:lctlboltdecorr}]
The first step of the proof is the following Lemma. Let $T_n =
\inf\ens{ k : S_{k+i} \ge S_k, \text{ for } i=1, \ldots,n}$. Clearly
$T_n < +\infty$ a.s.

\begin{lemma}\label{lem:astucebolthausen}
  For each sequences of real numbers $a_1, \ldots, a_n, \bar{a}_1,
  \ldots, \bar{a}_n$ we have
\begin{align}
\PP_0\big( S_k \le a_k, \bar{S}_k \le \bar{a}_k, 1\le k\le n &| S_k
\ge 0, 1\le k\le n\big) \\
\nonumber &= \bbP_0(S_{k+T_n} -S_{T_n} \le a_k,
\bar{S}_{k+T_n} -\bar{S}_{T_n} \le \bar{a}_k, 1\le k\le n).
\end{align}
\end{lemma}
\begin{proof}
  If $\cM_j := \cup_{s=0}^{s-1} \ens{S_s \le S_r, \text{ for } s+1 \le r
    \le \min(j , s+n)}$ we have 
$$ \ens{ T_n = j} = \cM_j^c \cap \ens{S_{j+k}\ge S_j, 1\le k \le n}.$$
What is interesting here is that $\cM_j \in\cF_j=\sigma(X_k, k\le j)$
and that the random walks $(S_{k+j} -S_j, k\ge 0)$ and $(\bar{S}_{k+j}
-\bar{S}_j, k\ge 0)$ are independent from the $\sigma$-field
$\cF_j$. Therefore
\begin{align*}
   &\prob{S_{k+T_n} -S_{T_n} \le a_k,
  \bar{S}_{k+T_n} -\bar{S}_{T_n} \le \bar{a}_k, 1\le k\le n} \\
&=\sum_{j=0}^{+\infty}  \prob{S_{k+j} -S_{j} \le a_k,
  \bar{S}_{k+j} -\bar{S}_{j} \le \bar{a}_k, 1\le k\le n \mid T_n
  =j} \prob{T_n =j} \\
&= \sum_{j=0}^{+\infty}  \prob{S_{k+j} -S_{j} \le a_k,
  \bar{S}_{k+j} -\bar{S}_{j} \le \bar{a}_k, 1\le k\le n \mid \cM_j^c
  ,S_{j+k}\ge S_j, 1\le k \le n} \prob{T_n =j}\\
&= \sum_{j=0}^{+\infty}\prob{ S_k \le a_k, \bar{S}_k \le \bar{a}_k, 1\le k\le n \mid S_k
  \ge 0, 1\le k\le n}\prob{T_n =j} \\
&=\prob{ S_k \le a_k, \bar{S}_k \le \bar{a}_k, 1\le k\le n \mid S_k
  \ge 0, 1\le k\le n}.
\end{align*}
\end{proof}
For $s\in (0,+\infty]$ let $C^s$ be the set of continuous functions on
$[0,s]$ (or $[0,+\infty)$ for $s=+\infty$) and $\cB^s$ the smallest
$\sigma$-algebra such that the mappings  $f\in C^s \to f(t) \in \R$
are measurable for every $t\in [0,s]$.

Let $P^s$ be the measure of Brownian motion on $(C^s,\cB^s)$. Then
$T^s : C^s \to [0,+\infty]$ is the mapping with
$$ T^s(f) = \inf\ens{t : f(u) \ge f(t), \text{ for } t\le u\le t+1 \le
  s}, \quad (\inf \emptyset = +\infty),$$
and we set $T= T^\infty$, $P= P^\infty$ for simplicity. Let $u$ be the
function in $C^1$ which is everywhere equal to $-1$. Let $\Phi_s:C^s\to
C^1$ be the map
$$ \Phi_s(f)(t) =
\begin{cases}
  f(T^s(f) + t) & \text{ for } T^s(f) <+\infty,\\
u & \text{otherwise.}
\end{cases}
$$

Let $Q_n$ be the probability measure defined on $(C^\infty \times
C^\infty, \cB^\infty \otimes \cB^\infty)$ by the process $(Y_n(t),
\bar{Y}_n(t), t\ge 0)$. Let $\pi_s : C^\infty \times
C^\infty \to C^s \times C^s$ be the projection map (we have $P^s
\otimes P^s =[P\otimes P] \circ \pi_s^{-1} $). Let $Q_n\circ\pi_1^{-1}(. \mid
\cC^+)$ be the probability measure defined on $C^1 \times C^1$ by
\begin{equation*}
  Q_n\circ\pi_1^{-1}(A \mid
\cC^+) := \frac{ Q_n (\pi_1^{-1}(A \cap \cC^+ \times C^1))}{Q_n
  (\pi_1^{-1}( \cC^+ \times C^1))} = \prob{ (Y_n(t),\bar{Y}_n(t), 0\le
  t\le 1) \in A \mid Y_n\in \cC^+ }.
\end{equation*}
In Lemma~\ref{lem:astucebolthausen} we have showed that $
  Q_n\circ\pi_1^{-1}(. \mid
\cC^+) = Q_n\circ \Psi^{-1}$,
with $\Psi$ defined as
\begin{equation*}
  \Psi_s(f,g)(t) =
\begin{cases}
  (f(T^s(f) + t) , g(T^s(f) + t))& \text{ for } T^s(f) <+\infty,\\
(u,u) & \text{otherwise.}
\end{cases}
\end{equation*}

As stated before, we can prove as in Donsker's theorem that  $Q_n\circ \pi_s^{-1}$
converges weakly (as $n \to \infty$) towards
 $ P^s \otimes P^s$.
Since $\Phi_s$ is continuous $P^s$ a.e. on $(C^s ,\rho )$ see \cite[Lemma 2.5]{B76}, we infer that $\Psi_s$ is continuous $P^s\otimes P^s$ a.e. on $(C^s \times
C^s,\rho \times \rho)$. By the contraction property, this implies that as $n\to \infty$
\begin{equation*}
  Q_n\circ (\Psi_s\circ\pi_s)^{-1}\quad \text{converges in distribution to} \quad [P^s\otimes P^s]\circ \Psi_s^{-1}.
\end{equation*}
We can now show as in \cite{B76} that if $\mathfrak{A}$ is a continuity set,
i.e. $[P\otimes P]\circ \Psi^{-1} (\partial \, \mathfrak{A})=0$, then
\begin{equation*}
  \lim_{n\to +\infty} Q_n\circ\Psi^{-1}(\mathfrak{A}) = [P\otimes P]\circ \Psi^{-1}(\mathfrak{A}),
\end{equation*}
and since $P\circ \Phi^{-1}$ is the law of $\cM^+$, we have that $ [P\otimes
P]\circ\Psi^{-1}$ is the law of $(\cM^+,B)$ and this concludes our proof.
    
  \end{proof}
\iflong

\section{Proof of Theorem \ref{thm:b}}\label{prthb}
\newcommand{\ttau}{{\tilde{\tau}}}

We start by showing Lemma \ref{lem:e:ridcond} which yields that 
it is equivalent to prove Theorem \ref{thm:b} with $S$ sampled from ${\bf P}_{\beta,\mu}(\cdot | \tau+G_{\tau-1}=L, S_\tau=0)$ rather than with $S$ 
sampled from ${\bf P}_{\beta,\mu}(\cdot | \tau+G_{\tau-1}=L)$.

\begin{lemma}\label{lem:e:ridcond}
  For every bounded random variable of  type $Z=F(\valabs{S}_{i\wedge \tau -1},
i\ge 1)$ we have 
$$ {\bf E}_{\beta,\mu}(Z \mid \tau + G_\ttau=L,S_\tau=0)  = {\bf P}_{\beta,\mu}(Z \mid
  \tau + G_{\tau -1}=L).$$
\end{lemma}
 \begin{proof}
Remember the geometric nature of the increments : ${\bf P}_{\beta,\mu}(S_1 = k ) =
c_\beta^{-1} e^{- \frac{\beta}{2} \valabs{k}}$ so that there exists
$C>0$ such that
$$ \forall x \ge 1,\quad C\,  {\bf P}_{\beta,x}( S_1\le 0) = {\bf P}_{\beta, x}(S_1=0).$$
Therefore, by symmetry,
\begin{align*}
& {\bf E}_{\beta,\mu}(Z ; \ttau + G_\ttau=L,S_\ttau=0) \\ 
  &= \sum_N {\bf E}_{\beta,\mu}(Z ; \ttau + G_{\ttau-1}=L,S_\ttau=0,\ttau=N)\\
 & = 2 \sum_N \sum_{k=0}^{N-2}{\bf E}_{\beta,\mu}(Z;S_1=\cdots=S_k=0,S_{k+1}>0, \ldots, S_{N-1}>0,S_N=0, N+ A_{N-1}=L)\\
  &=2 \sum_N \sum_{k=0}^{N-2}\sum_{b\ge 1} {\bf E}_{\beta,\mu}(Z;S_1=\cdots=S_k=0,S_{k+1}>0, \ldots, S_{N-1}=b, N+ A_{N-1}=L) {\bf P}_{\beta,b}(\etp{S_1=0})\\
  &=2C \sum_N \sum_{k=0}^{N-2}\sum_{b\ge 1} {\bf E}_{\beta,\mu} (Z;S_1=\cdots=S_k=0,S_{k+1}>0, \ldots, S_{N-1}=b, N+ A_{N-1}=L) {\bf P}_{\beta,b} (S_1\le 0)\\
&  = C  {\bf E}_{\beta,\mu}(Z;\ttau + G_{\ttau-1}=L).
\end{align*}
\end{proof}
As for Theorem \ref{thm:ab}, the proof of Theorem \ref{thm:b} is very close to that of Theorem \ref{thm:a}. For this reason, we will not display 
the proof of theorem \ref{thm:b} in full details here, but we will show how to derive the counterpart of Proposition \ref{pro:llttausupncomplet} with $\tau$ instead of $\ttau$,
with the underlying probability measure ${\bf P}_{\beta,\mu}$ instead of the generic $\PP$ and with $G_N$ instead of $A_N$.
This is indeed  the key tool to obtain the convergence in finite dimensional distribution which is required to prove Theorem \ref{thm:b}.  The next steps : inverting the conditioning, establishing tightness,
  performing time change, are so similar to those taken in section
  \ref{sec:preuvea} that we feel free to omit them.

\begin{proposition}\label{pro:b:lltexcursion} 
 With $r= d + \frac32$  and $[h_1,h_2]\subset (0,\infty)$
\begin{align}\label{propexc3}   
\lim_{N\to +\infty}  \sup_{a\in [h_1,h_2] N^{3/2}}  \sup_{(x,y)\in \N^d\times  \Z^{d}} \Big |
N^r  {\bf P}_{\beta,\mu}(&G_N=a,\, 
  \snt=x, \bar{\mathbf S}_{\floor{Nt}}=y \mid  \tau = N,\, S_N=0) \\
 \nonumber & - \frac{1}{\sigma_\beta^{r}} \phi_t\etp{\frac{a}{\sigma_\beta
    N^{3/2}}, \frac{x}{\sigma_\beta N^{1/2}}} k\etp{\frac{y}{\sigma_\beta N^{1/2}}}
\Big |=0.
\end{align}
\end{proposition}

%

  Our stategy of proof for Proposition \ref{pro:b:lltexcursion} is to apply
  Proposition~\ref{thm:cvloiplusapproxregegallclt}. Therefore, we
  shall first establish the necessary convergence in distribution in
  subsection
  \ref{subsec:b:cvdis} and then smoothness of the approximation in
  subsection \ref{subsec:b:smooth}.

  \subsection{Convergence in distribution for Proposition  \ref{pro:b:lltexcursion}} \label{subsec:b:cvdis}
    
We pick $t=(t_1,\dots,t_d)$ with $0<t_1<\dots<t_d\leq 1$ and we shall prove that when sampled under ${\bf P}_{\beta,\mu}(\cdot\,  |\, \tau=N, S_N=0)$ the vector
$\frac{1}{\sigma_\beta}\Big(\frac{G_N}{N^{3/2}}, \frac{\valabs{\mathbf{S}_{\floor{tN}}}}{N^{1/2}}, \frac{\bar {\mathbf{S}}_{\floor{tN}}}{N^{1/2}}\Big)$ 
converges in distribution to $(A(e),\mathbf{e}_{t},\mathbf B_{t})$ with $\mathbf e_t:=(e_{t_1},\dots,e_{t_d})$ and   with $B$ a standard Brownian motion and $e$ a standard Brownian excursion independent of $B$. For simplicity we will abuse notation and omit the $\frac{1}{\sigma_\beta}$ in fromt of 
$\Big(\frac{G_N}{N^{3/2}}, \frac{\valabs{\mathbf{S}_{\floor{tN}}}}{N^{1/2}}, \frac{\bar{\mathbf S}_{\floor{tN}}}{N^{1/2}}\Big)$.  Since we proved in \cite[Appendix 7.2]{CP15} that
  \begin{equation}\label{eq:b:asympatus}
    {\bf P}_{\beta,\mu}(\ttau=N,S_N=0) \sim C_2  {\bf P}_{\beta,\mu}(\tau=N,S_N=0) \sim C_3 N^{-3/2}\,,
  \end{equation}
the convergence in distribution will be established once we show that
there exists $C'_3>0$
such that for every bounded continuous function
$F$,

\begin{equation}
  \lim_{N\to +\infty}N^{3/2}\,   {\bf E}_{\beta,\mu}\bigg[F\bigg(\frac{G_N}{N^{3/2}},\frac{\valabs{\mathbf{S}_{\floor{tN}}}}{N^{1/2}}, \frac{\bar{\mathbf S}_{\floor{tN}}}{N^{1/2}}\bigg); \tau=N,S_N=0\bigg] = C'_3 \esp{F(A(e),\mathbf e_t,\mathbf B_t)}\,,
\end{equation}
since letting $F=1$ in the preceding equation shows that $C'_3=C_3$.
 
 When the random walk starts from $b\ge 1$ then $\ttau=\tau$ and $| S_{\floor{Ns}}|=S_{\floor{Ns}}$ for $s\in [0,1]$. If, moreover $S_\tau=0$ then $G_{\tau}=A_\tau$. Therefore
conditioning by the value of $S_0$ and taking into account symmetry,
we have 
\begin{equation}\label{decDD}
   {\bf E}_{\beta,\mu}\bigg[F\etp{\frac{G_N}{N^{3/2}},\frac{\valabs{\mathbf{S}_{\floor{tN}}}}{N^{1/2}},  \frac{\bar{\mathbf S}_{\floor{tN}}}{N^{1/2}}}; \tau=N,S_N=0\bigg] = D_{1,N} + D_{2,N}
\end{equation}
with
\begin{align*}
  D_{1,N} &= \mu(0)\,  {\bf E}_{\beta,0}\bigg[F\etp{\frac{G_N}{N^{3/2}},\frac{\valabs{\mathbf{S}_{\floor{tN}}}}{N^{1/2}},  \frac{\bar{\mathbf S}_{\floor{tN}}}{N^{1/2}}}; \tau=N,S_N=0\bigg]\,,\\
  D_{2,N} &= 2 \sum_{b\ge 1} \mu(b) \, {\bf E}_{\beta,b}\bigg[F\etp{\frac{A_N}{N^{3/2}},\frac{\mathbf{S}_{\floor{tN}}}{N^{1/2}},  \frac{\bar{\mathbf S}_{\floor{tN}}}{N^{1/2}}}; \ttau=N,S_N=0\bigg].
\end{align*}

To handle $D_{2,N}$ we condition on the value of $S_1$, use symmetry
and the fact that there exists $C_4>0$ such that $\mu(b) = C_4 \probbeta{S_1=b}$ (for $b\in \N$)  to write
\begin{align*}
  D_{2,N} &= 2 C_4 \sum_{b\ge 1} {\bf P}_{\beta}(S_1=b)\,  {\bf E}_{\beta,b}\bigg[F\etp{\frac{A_N}{N^{3/2}},\frac{\mathbf{S}_{\floor{tN}}}{N^{1/2}},  \frac{\bar{ \mathbf{S}}_{\floor{tN}}}{N^{1/2}}}; \ttau=N,S_N=0\bigg]\\\
    &= 2 C_4 {\bf E}_{\beta}\bigg[F\etp{\frac{A_{N+1}}{N^{3/2}},\frac{{\mathbf{S}_{1+\floor{tN}}}}{N^{1/2}}, \frac{{\bar{\mathbf{S}}_{1+\floor{tN}}}}{N^{1/2}}};\ttau=N+1,S_{N+1}=0\bigg]
\end{align*}
We then use \eqref{eq:b:asympatus} in combination with  the convergence in distribution already referred to in the first paragraph of 
Subsection \ref{subsec:a:lltexc} as a consequence of \cite[Corollary
2.5]{CarCha13},  to obtain that there exists a $C_5>0$ such that 
\begin{equation}
  \lim_{N\to +\infty}  N^{3/2} D_{2,N} = 2 C_5 \esp{F(A(e),\mathbf e_t, \mathbf B_t)}.
\end{equation}
To handle $D_{1,N}$ we take into account the time the random walk
sticks to $0$ and obtain
\begin{align*}
  \frac{D_{1,N} }{2\mu(0)} &= \sum_{0\le k\le N-2} {\bf E}_{\beta}\bigg[F\etp{\frac{G_N}{N^{3/2}},\frac{\valabs{\mathbf{S}_{\floor{tN}}}}{N^{1/2}},\frac{\bar{\mathbf{S}}_{\floor{tN}}}{N^{1/2}}}; S_1=\cdots=S_k=0,S_{k+1}>0,\ldots,S_{N-1}>0,S_{N}=0\bigg] \\
  &=  \sum_{0\le k\le N-2} \frac{1}{c_\beta^k} {\bf E}_{\beta}\bigg[F\etp{\frac{A_{N-k}}{N^{3/2}},\frac{{\mathbf{S}_{\floor{tN}-k}}}{N^{1/2}},\frac{\bar{\mathbf{S}}_{\floor{tN}}}{N^{1/2}}};\ttau=N-k,S_{N-k}=0\bigg] =: \sum_{0\le k\le N-2} b_{k,N}\,.
\end{align*}
We shall conclude by applying the dominated convergence
theorem. First, the convergence in distribution referred to in
subsection \ref{subsec:a:lltexc} ensures us that for every $k\in \N$, 
$$ \lim_{N\to +\infty} N^{3/2}\, b_{k,N} = \frac{C_5}{c_\beta^k} \esp{F(A(e),\mathbf e_t,\mathbf B_t)}\,.$$
Second, since $c_\beta>1$, the domination bound is easily derived from
\ref{eq:b:asympatus} which implies that there exist $C_6, C_7>0$ such that  
\begin{align*}
  N^{3/2}\, b_{k,N} &\le C_3 \norme{F}_{\infty}\etp{ \frac{N}{N-k}}^{3/2} c_\beta^{-k}\\
  &\le C_6 \etp{ \frac{N}{N-k}}^{3/2} c_\beta^{-k}\,  ( \ind_{\{k< N/2\}}+\ind_{\{k\ge N/2\}} ) \\
  &\le C_6 \bigg(2^{3/2} c_\beta^{-k} + c_\beta^{-k/2} \, \frac{N^{3/2}}{c_\beta^{k/2}}\, \ind_{\{k\ge N/2\}}\bigg) \le C_7\,  c_\beta^{-k/2}.
\end{align*}

\subsection{Smoothness of approximations} \label{subsec:b:smooth}

A tedious but straightforward way to obtain items (i) to (iv) of
Proposition~\ref{thm:cvloiplusapproxregegallclt}, is to follow the
same steps as for the proof of Theorem \ref{thm:a}, substituting the conditioning
on $\tau$ to the conditioning on $\ttau$. We will not repeat all those steps here, but we will 
show (with lemmas \ref{balt1} and \ref{balt2} below) how to adapt the proofs of Section \ref{subsec:a:lltexc} to derive the counterparts bounds of \eqref{eq:simtau} and \eqref{eq:crudeboudsnxtau}. 
\begin{lemma}\label{balt1}
There exists $\widetilde C_1>0$ such that for 
  $${\bf P}_{\beta,\mu}(\tau>N)= \frac{\tilde{C}_1}{N^{-1/2}} (1+o(1)).$$
\end{lemma}
\begin{proof}
The same decomposition as that used to obtain \eqref{decDD}
allows us to write  
  \begin{align}\label{dercal}
  \nonumber  {\bf P}_{\beta,\mu}(\tau>N) &= \mu(0)\,  {\bf P}_{\beta}(\tau >N) + 2\,  \sum_{b\ge 1} \mu(b)\,  {\bf P}_{\beta,b}(\tilde \tau >N) \\
    &= \mu(0)\,  {\bf P}_{\beta}(\tau >N)+ 2\,  C_4\,  {\bf P}_{\beta}(\ttau > N+1).
  \end{align}
  We now compute ${\bf P}_{\beta}(\tau >N)$ by taking into account the time $k$ that the random walk
  sticks to $0$
  \begin{align*}
{\bf P}_{\beta}(\tau >N) &= {\bf P}_{\beta}(S_1= \cdots=S_N=0) + 2 \sum_{0\le k \le N-1} {\bf P}_{\beta}(S_1= \cdots=S_k=0, S_{k+1}>0,\ldots,S_N>0) \\
    &= \frac{1}{c_{\beta}^{N}}  + 2 \sum_{0\le k \le N-1} \frac{1}{c_{\beta}^{k}}\,  {\bf P}_{\beta}(\ttau >N-k)
  \end{align*}
  With \eqref{eq:simtau} and since $c_\beta>1$, we can use the dominated convergence theorem to show that
  $$ N^{1/2} \sum_{0\le k \le N-1} \frac{1}{c_{\beta}^{k}}  {\bf P}_{\beta}(\ttau >N-k)
  \to C_1 {\sum_{k\ge 1} \frac{1}{c_{\beta}^{k}}}<\infty$$
  and this ends the proof.
\end{proof}

We can establish exactly in the same way that there exists $\widetilde C_2>0$ such that  ${\bf P}_{\beta,\mu}(\tau = N)=
\widetilde C_2 N^{-3/2} (1+o(1))$ and we also recall \eqref{eq:b:asympatus}.
Let us write explicitly another bound.
\begin{lemma}\label{balt2}
There exists $\widetilde C_3>0$ such that for every $N\in \N$
  \begin{equation*}
    \sup_{x\in \Z} {\bf P}_{\beta,\mu}(S_N=x,\tau>N) \le \frac{ \tilde{C}_3}{ N}\,.
  \end{equation*}
\end{lemma}
\begin{proof}
   We use Markov's property at time $M=\floor{N/2}$ and the
  crude bound \eqref{eq:simtau2} to get
  \begin{align*}
  {\bf P}_{\beta,\mu}(S_N=x,\tau>N) &= \sum_z {\bf P}_{\beta,\mu}(S_M=z,\tau >M)\,  {\bf P}_{\beta,z}(S_{N-M}=x,\tau> N-M)\\
    &\le \sum_z  {\bf P}_{\beta,\mu}(S_M=z,\tau >M)\,  {\bf P}_{\beta,z}(S_{N-M}=x) \\
    &\le \frac{C_2}{ (N-M)^{1/2}} \sum_z  {\bf P}_{\beta,\mu}(S_M=z,\tau >M) = \frac{ C_2}{ (N-M)^{1/2}} {\bf P}_{\beta,\mu}(\tau >M) \\
    &\le \frac{C_2 \tilde{C}_1}{N}.
  \end{align*}
  \end{proof}
  
The analogues of all bounds in Section \ref{subsec:a:lltexc} (i.e., (\ref{eq:simtau}--\ref{eq:crudeboudsnxtau}) or (\ref{eq:tsnxpetit}--\ref{eq:e:lltboundxsmall}) and (\ref{eq:smoothnessstandartllt}--\ref{eq:smoothnessstandartllt22})  
can be derived similarly with $\tau$ instead of $\ttau$.

\fi

\section{Excursion measure normalized by its length or area}\label{excnorm2}
\newcommand{\upfad}{(1+ \frac\alpha2)}
\newcommand{\usupfad}{\unsur{1+ \frac\alpha2}}
\newcommand{\nrho}[1]{n^\rho\etp{#1}}
\newcommand{\nplus}[1]{n_+\etp{#1}}
\newcommand{\intof}{\int_0^{+\infty}}

The law of a normalized excursion may be defined either by the
independence property (Proposition~\ref{pro:indeshaperho}) or the
disintegration property (Proposition~\ref{pro:equivshapedisin}).

We combine these results with  the well known fact
that the power of a Bessel process is a time changed Bessel Process to
establish finally Theorem~\ref{thm:c}.
\subsection{Normalizing an excursion measure}
Let $\rho$ be a Bessel process of dimension $\delta\in(0,2)$ and
$n^\rho$ be its Itô's measure for excursions away from $0$. The
following description of $n^\rho$ has been established in
\cite[Theorem 1.1]{PitYor96} (and generalizes Itô's description for
  Brownian motion $\delta=1$):
  \begin{itemize}
  \item The lifetime $R$ of the excursion has the $\sigma$-finite
    density: with $\delta'=4-\delta$,
    \begin{equation}
      \label{eq:denszetabessel}
      n^\rho(R \in dt) = 2^{- \delta'/2} \Gamma(\delta'/2)^{-1}
      t^{-\delta'/2}\, dt.
    \end{equation}
\item for any positive measurable functional $F$ defined on the
  excursion space $(U,\cU)$ we have the disintegration of measure
  \begin{equation}
\label{eq:disintegrationbessel}
n^\rho(F) = \int_0^{+\infty} P^{\delta',t}_{0,0}(F) n^\rho(R \in dt),
\end{equation}
with $ P^{\delta',t}_{0,0}$ the law of a Bessel bridge of dimension
$\delta'$ and length $t$.
  \end{itemize}

This description leads to defining $ P^{\delta',1}_{0,0}$ as the law
of an excursion of $\rho$ normalized by its length (lifetime). 
We shall see that thanks to the Brownian scaling we can extend this to
other functionals than $R$, by establishing an independence
property equivalent to the disintegration of measure.

Let $\alpha \ge 0$ and for $w\in U$, a generic excursion of lifetime
$$R(w) =\inf\ens{t>0 : w(t)=0},$$
we consider the weighted area
\begin{equation}
  \label{eq:defweightedarea}
  A^\alpha_t(w)=A_t(w) = \int_0^t w(s)^\alpha\, ds,\quad
  A(w):=A_\infty(w)= \int_0^{R} w(s)^\alpha\, ds.
\end{equation}
Let $s_c:U\to U$ be the Brownian scaling operator : 
$\displaystyle s_c(w)(t) = \unsur{\sqrt{c}} w(ct)$ for $t\ge 0$.
Then 
$\displaystyle A(s_c(w)) = c^{-\upfad} A(w)$
 and we define the normalizing by $A$ operator as
\be{dispsa}
\nu(w)= s_{A(w)^{\usupfad}}(w),
\ee
which satisfies $A(\nu(w))=1$.

\begin{proposition}\label{pro:indeshaperho}
  Assume that $n^\rho(A>1) < +\infty$. Then there exists a probability
measure $\pi^{\rho,A}$ defined on $(U,\cU)$ such that for every
positive measurable $F,\psi$ :
\begin{equation}
  \label{eq:indeshape}
  n^\rho(F\circ \nu \, \psi(A)) = \pi^{\rho,A}(F)\,   n^\rho(\psi(A)).
\end{equation}
Moreover $\pi^{\rho,A}(A=1)=1$.
\end{proposition}
\begin{remark}
Relation \eqref{eq:indeshape} is called the \emph{independence
  property} since it shows that the shape of an excursion $w$ is
independent of the value of the functional $A(w)$. We call
$\pi^{\rho,A}$ the \emph{law of an excursion of the Bessel process $\rho$
normalized by its functional $A$}. When $\alpha=0$, $A(w)=R(w)$
and $\pi^{\rho,R}=\pi^\rho$ is the law of normalized excursion of $\rho$.
\end{remark}

\begin{proof*}
  By scaling of the Bessel process we have $n^\rho(F(s_c(w)))=
  \unsur{\sqrt{c}} n^\rho(F(w))$. Therefore, if $c>0$ and
  $t=c^{-\upfad}$ :
$$ n^\rho(A>t)= \sqrt{c}\, n^\rho(A c^{-\upfad} >t) = \sqrt{c}\,
n^\rho(A>1) < +\infty.$$
Since $A$ is non identically zero we also have $n^\rho(A>t)>0$.

We claim that the measure
$$\pi_t (F) :=\frac{\nrho{F\circ \nu ; A \ge
    t}}{\nrho{A\ge t}}$$
 does not depend on the value of  $t>0$, since by scaling
 $\nu(s_c(w))=\nu(w)$ and
$$ c^{-\undemi} \nrho{F\circ \nu; A\ge t} = \nrho{F\circ \nu\circ s_c
; A\circ s_c \ge t} = \nrho{F\circ \nu; c^{-\upfad} \, A\ge t}.$$

Therefore $\pi_t(F)=\pi_{tc^{\upfad}}(F)$. We define
$\pi^{\rho,A}:=\pi_1$. It is a probability (take $F\equiv 1$) and
since $A(\nu(w))=1$ we have
$$ \pi^{\rho,A}(A=1)= \frac{\nrho{A\circ \nu=1; A\ge t}}{\nrho{A\ge t}}=1.$$

We now generalize this identity to \eqref{eq:indeshape} by a functional monotone class theorem
as follows. First observe that it suffices to prove that for any $c>0$
we have
\begin{equation}
  \label{eq:indeshape2c}
  n^\rho(F\circ \nu \, \psi(A); A>c) = \pi^{\rho,A}(F)\,  n^\rho(\psi(A);
  A>c),
\end{equation}
and use monotone convergence letting $c\downarrow 0^+$ to obtain
\eqref{eq:indeshape}.

Now fix for the moment $F$ positive measurable and bounded so that
$\pi^{\rho,A}({F}) < +\infty$. The space $\cH$ of bounded measurable
$\psi$ satisfying \eqref{eq:indeshape2c} is a monotone space that
contains $\cC:=\ens{1_{[a,+\infty[}, a\ge 0}$ stable by finite
products. Therefore $\cH$ contains all bounded $\sigma(\cC)$ measurable
$\psi$, i.e. all bounded Borel measurable $\psi$.
We can go now from positive bounded $F$ to positive $F$ by fixing
$\psi$ positive measurable and use monotone convergence. \hfill$\square$

\end{proof*}

Let us state the equivalence between the disintegration of measure
and the shape independence property.
\begin{proposition}\label{pro:equivshapedisin}
  There exists a probability $\pi$ on $(U,\cU)$ such that for all
  positive measurable $F$:
$$\nrho{F} = \int_0^{+\infty} \pi_a(F) \nrho{A\in da},$$
with $\displaystyle \pi_a=\pi\circ (s_{c})^{-1}$ image of $\pi$ by
$s_{c}$ with $c=a^{-\usupfad}$ , if and only if there exists a probability measure
$\mu$ on $(U,\cU)$ such that for every positive measurable $F,\psi$:
\begin{equation}\label{eq:indepmu}
\nrho{F\circ\nu\, \psi(A)} = \mu(F) \, \nrho{\psi(A)}.
\end{equation}
Moreover, if this is the case we have $\pi=\mu$ and $\pi(A=1)=1$.
\end{proposition}
The proof of Proposition \ref{pro:equivshapedisin} is inspired by
\cite{PitYor96} and uses the links between scaling of Bessel processes
and independence.

\iflong
\begin{proof}
  \emph{Necessity} Assume we have the disintegration. Since $\nu\circ
  s_c=\nu$, we have
  \begin{align}
    \nrho{F\circ\nu\, \psi(A)} &= \intof \pi_a(F\circ\nu \psi(A))\,
    \nrho{A\in da} \notag\\
&= \intof \pi\etp{F\circ\nu(s_{a^{-\usupfad}}(w))\,
  \psi(A(s_{a^{-\usupfad}}(w)))}  \nrho{A\in da}\notag \\
&= \intof \pi\etp{F\circ\nu(w)\,
  \psi(aA)}  \nrho{A\in da} \,. \label{eq:shdistrois}
  \end{align}
Letting $F\equiv 1$ we obtain for all positive measurable $\psi$:
$$ \nrho{\psi(A)} = \intof \pi\etp{  \psi(aA)}  \nrho{A\in da} =
\intof \psi(a)  \nrho{A\in da}\,.$$
Therefore for almost every $a>0$, $\pi(\psi(a A))=\psi(a)$ that is
$\pi(A=1)=1$ so that under $\pi$, $\nu(w)=w$ and we reinject this in
\eqref{eq:shdistrois} to obtain
 \begin{align*}
    \nrho{F\circ\nu\, \psi(A)} &= \intof \pi(F\circ\nu ) \psi(a)
    \nrho{A\in da} \\
&= \intof \pi\etp{F} \psi(a)  \nrho{A\in da} \\
&=\pi(F)\,  \nrho{\psi(A)} \,. 
  \end{align*}
\emph{Sufficiency}. Assume now the existence of $\mu$
satisfying\eqref{eq:indepmu}. Then by the monotone class theorem we
get that for every positive measurable $G$ we have

$$ \nrho{G(\nu(w), A(w))} = \int\int G(w,a) d\mu(w) \nrho{A\in
  da}\,.$$
Therefore if we let $\pi=\mu$ and $\pi_a= \pi\circ (s_{c})^{-1}$,
we obtain for a positive measurable $F$, letting
$G(w,a):=F(s_{a^{-\usupfad}}(w))$,
$$ \nrho{F}=\nrho{G(\nu(w), A(w))} = \intof \etp{\int
  F(s_{a^{-\usupfad}}(w)) \pi(dw)} \nrho{A\in
  da} = \intof \pi_a(F) \nrho{A\in
  da}\,.
$$
\end{proof}

\fi

Combining \eqref{eq:denszetabessel} and \eqref{eq:disintegrationbessel} we obtain the 
\begin{corollary}\label{cor:excmeabessellength}
  The law $\pi^\rho$ of a Bessel excursion normalized by its length is
  well defined. We have $\pi^\rho=P^{\delta',1}_{0,0}$ and the
  independence property for the normalizing operator
  $\nu(w)=s_{R(w)}(w)$ is thus
$$ \nrho{F(\nu(w))\,  \psi(R(w)} = \pi^\rho(F)\,  \nrho{\psi(R)}.$$
\end{corollary}

For Brownian motion, $\delta=1$, we have $\pi^\rho=P^{3,1}_{0,0}$ and
we let $(e_t,0\le t\le 1)$ be a process distributed as
$P^{3,1}_{0,0}$, and named (normalized) Brownian excursion.

Another important application takes place for $\alpha=1$, that is
$A(w)=\intof w(s)\, ds$ is the excursion area, and $\delta=1$ that is
$\rho=\valabs{B}$ is the absolute value of a Brownian motion. There
$n^\rho=2n_+$ with $n_+$ the positive excursion measure of Brownian
excursions away from zero.
\begin{corollary} \label{cor:defnormedexcursionarea} For $\delta=\alpha=1$, 
  the law $\pi^{\rho,A}$ of the Brownian excursion normalized by its
  area $(\cE_s, s\geq 0)$ is well defined. If $\nu^1(w)=s_{A(w)^{2/3}}(w)$ is the
  normalizing operator, we have
  \begin{gather*}
    \nplus{F(\nu^1(w)) \psi(A(w))}=\esp{F(\cE)} \nplus{\psi(A)}\,,\\
    \nplus{F} = \int_0^\infty \esp{F\etp{a^{1/3} \cE_{a^{-2/3}s},0\le
        s\le 1}} \, \nplus{A\in da}\,.
  \end{gather*}
\end{corollary}
\begin{proof}
  It suffices to prove that $\nplus{A>1} < +\infty$. First observe
  that if $\tau_t$ is the inverse local time of Brownian motion, we
  have by the exponential formula of excursions, on the one hand, for
  any $\lambda>0$, 
$$ \esp{e^{-\lambda A_{\tau_t}}} = \exp\etp{- 2 t \int (1-e^{-\lambda
    a})\, \nplus{A\in da}}.$$
\iflong
On the other hand, by Ray-Knight theorem, if $(X_t,t\ge 0)$ is a square
Bessel process of dimension $0$ starting form $t$, we have
$$ \esp{e^{-\lambda A_{\tau_t}}} = \esp{e^{-\lambda \intof t
    X_t\, ds}}^2= e^{t \phi'(0)}$$
with $\phi(t)$ the unique positive non-increasing solution of
$\phi''(t)=\phi(t) \lambda t$ such that $\phi(0)=1$ (see \cite[Chapter
XII]{RY91}). We recognize an Airy-type  equation, $\phi(t)= \alpha
Ai(\lambda^{1/3} t) + \beta Bi(\lambda^{1/3} t)$ and the boundedness
impose $\phi(t)=\alpha Ai(\lambda^{1/3} t)= \unsur{Ai(0)}
Ai(\lambda^{1/3} t)$ since $\phi(0)=1$. We get $\phi'(0)=-C
\lambda^{1/3}$ and thus, for all $\lambda >0$, by Fubini's
$$  C \lambda^{1/3}=\int (1-e^{-\lambda a}) \nplus{A \in da} =\lambda
\int e^{-\lambda a} \nplus{A>a}\, da\,.$$
\fi
Inverting the Laplace transform yields
$\displaystyle\nplus{A>a} = C' a^{-1/3} < +\infty$.
\end{proof}

\subsection{Absolute continuity relationship between distributions of
  Brownian excursion normalized by its length and Brownian excursion normalized by its area}

This absolute continuity will be established by playing with  the
independence property of normalized excursions.

\begin{lemma}\label{lem:norlennorarea}
Let $h(u,x)$ be the joint density of $(A(e); e_{t_i}, 1\le i\le
d)$. Then there exists a constant $C>0$ such that the function 
$$ g(u,x) = C u^{-3} h(u^{-3/2}, x)$$
is the density of $(R(\cE), \unsur{\sqrt{R(\cE)}} \cE_{t_i
R(\cE)}, 1\le i\le d)$.
\begin{proof}
  
Let  $\phi(u,x)$ be measurable non-negative and 
\begin{align}
 I_\phi &:=  \int_0^{+\infty} du \int_{\R^d} \phi(u,x) g(u,x) dx
= C_1 \int_0^{+\infty} dy \int_{\R^d} \phi(y^{-2/3},x)y^{1/3} h(y,x)
dx \\
&=  C_1
\esp{\phi(A(e)^{-2/3}; e_{t_i},1\le i\le d) A(e)^{1/3}}.
\end{align}
To identify $I_\phi$ we shall use the independence property for the
excursion length : recall that $\nu(w)(t) = \unsur{\sqrt{R(w)}}
w(t R(w))$, for any positive measurable $\psi$:
\begin{align*}
  J_{\phi,\psi} &:= n_+(\psi(R))\; \esp{\phi(A(e)^{-2/3};
    e_{t_i},1\le i\le d) A(e)^{1/3}}\\
&= \nplus{A\circ \nu(w)^{1/3} \psi(R) \phi(A\circ \nu(w)^{-2/3};
  \nu(w)(t_i), 1\le i\le d)}\\
&= \nplus{A^{1/3} R^{-1/2} \psi(R) \phi(A^{-2/3} R;
  \nu(w)(t_i), 1\le i\le d)}.
\end{align*}
Now we use the independence property for the area. Recall that the
normalizing operator is $u(t):=\nu^1(w(t))= A(w)^{-1/3}
w(tA(w)^{2/3})$ and observe that we have $A(u)=1$, $R(u)=A(w)^{-2/3}
R(w)$ and
$ \nu(w)(t) = \nu(u)(t)$. Therefore
\begin{align*}
  J_{\phi,\psi} &:=\nplus{R(u)^{-1/2} \psi(R(u) A(w)^{2/3})
    \phi(R(u); \nu(u)(t_i), 1\le i\le d)} \\
&= \int \int R(u)^{-1/2} \psi(R(u) a^{2/3})
    \phi(R(u); \nu(u)(t_i), 1\le i\le d) \prob{\cE\in du} \nplus{A
      \in da} \\
&=\esp{\phi(R(\cE);\nu(\cE)(t_i), 1\le i\le d)\; \gamma(R(\cE))},
\end{align*}
with $\gamma(t):=t^{-1/2}\nplus{\psi(t A^{2/3}}$. We now identify this
function $\gamma(t)$. Recall that 
$$ \nplus{A\in da}=C_2 a^{-4/3} da,\quad \nplus{R\in dt}= C_3
t^{-3/2} dt.$$
We obtain
$$ \gamma(t) = C_2 \intof \psi(t a^{2/3})\, da = C_4 \intof \psi(u)
u^{-3/2}\, du = C_5 \nplus{\psi(R)}.$$
Injecting this into the last expression of $J_{\phi,\psi}$ yields
$$ J_{\phi,\psi} = C_5 \nplus{\psi(R) }
\esp{\phi(R(\cE);\nu(\cE)(t_i), 1\le i\le d)}.$$
Comparing the two expressions of $J_{\phi,\psi}$, we conclude that
$\displaystyle I_\phi = C_{6} \esp{\phi(R(\cE);\nu(\cE)(t_i), 1\le i\le d)}$
\end{proof}

\end{lemma}

\subsection{Applying a time change to obtain the excursion measure of $Y_t=\big|B_{a_t}\big|$}

First, if we consider the excursion measure of $\valabs{B}$ with
respect to its local time at level zero 
$L^0_t(\valabs{B})= 2 L^0_t(B)$
of right continuous inverse  $\tau^{\valabs{B}}_t=\tau_{t/2}$ we
obtain that 
$n^{\valabs{B}}=n_+$ thanks to the exponential formula. 
\iflong 
Let  $f$ be measurable non-negative,
\begin{align*}
  \esp{\exp(- \sum_{0<s\le t} f(s,e^{\valabs{B}}_s))} &= \esp{\exp(-
    \sum_{0<s \le t} f(s,\valabs{e_{s/2}}))} 
=\exp\etp{- \int_0^{t/2} ds \int (1-e^{-f(2s,\valabs{u})})\, n(du)}\\
&\hspace{-1cm}=\exp\etp{-\undemi \int_0^t dr \int (1-e^{-f(r,\valabs{u})})n(du)} 
=\exp\etp{- \int_0^t dr  \int (1-e^{-f(r,u)}) n_+(du)}.
\end{align*}
\fi

We consider here the area of an excursion $A_t(w) = \int_0^t w(s)\,
ds$ and its inverse $a_w(s):=\inf\ens{t>0 : A_t(w)>s}$.

\begin{lemma}\label{lem:excy} The excursion measure away from zero  
   $n^Y(dw)$ of  $Y_t=\valabs{B_{a_t}}$ is the image of 
  $n_+(dw)$ by the function $(w(t),t\ge 0) \to (w(a_w(t)), t\ge 0)$.
\end{lemma}
\begin{proof}
  \iflong
We shall give a proof based on the Master formula.
 let
  $L_t,\tau_t$ be the local time at level $0$ and its inverse for 
  $\valabs{B}$. Then $s\in G_\omega^Y$ i.e. $s$ is the left endpoint
  of an interval where $ Y(\omega)$ does not vanish iff $a_s \in
  G_\omega^{\valabs{B}}$ and thus if there exists  $u>0$ such that
  $a_s=\tau_{u-}$ that is  $s=A_{\tau_{u-}}$. Then the excursion of 
  $Y$ starting from  $s$ is 
$$ i^Y_s(\omega)(r)= Y_{s+r}(\omega)=\valabs{B}(a(s+r)) =
\valabs{B}(a(A_{\tau_{u-}}+r) - \tau_{u-} + \tau_{u-}) =
i^{\valabs{B}}_{\tau_{u-}}(a(A_{\tau_{u-}}+r) - \tau_{u-})$$
and we have
$$ a(A_{\tau_{u-}}+r) - \tau_{u-} = \inf\ens{t: A_t > A_{\tau_{u-}}+r}
-\tau_{u-} = \inf\ens{t: A_{t+\tau_{u-}} > A_{\tau_{u-}}+r} =
a_{i^{\valabs{B}}_{\tau_{u-}}}(r).$$
This is the  inverse of the area of the excursion of  $\valabs{B}$
starting from  $\tau_{u-}$. By the  \emph{Master Formula} (see
 \cite[Chapter XII]{RY91}) for
 $\valabs{B}$ :
 \begin{align*}
\esp{\sum_{s\in G_\omega^Y} H(s,\omega,i^Y_s(\omega))} &=\esp{\sum_u
  H(A_{\tau_{u-}}(\omega),\omega;
  i^{\valabs{B}}_{\tau_{u-}}(a_{i^{\valabs{B}}_{\tau_{u-}}(.)}))} \\
&= \esp{\sum_u  H(A_{\tau_{u-}}(\omega),\omega; e_u(a_{e_u}(.)))} \\
&= \esp{\int_0^\infty ds \int H(A_{\tau_s}(\omega),\omega; w(a_w(.))) \,
  n_+(dw)}\\
&=\esp{\int_0^\infty ds \int H(\tau^Y_s,\omega; w(a_w(.))) n_+(dw)}.
 \end{align*}
This is the master formula for $Y$, whence the result.
\fi
\end{proof}

\subsection{Proof of Theorem \ref{thm:c}}
\label{subsec:besproof}
A power of a Bessel process is up to a constant another Bessel
process, with a change of time (see \cite[Chapter XI, Proposition
(1.11)]{RY91} ). Here we get that 
$$ \rho_t := \frac23 Y_t^{3/2}=\frac23 \valabs{B_{a_t}}^{3/2}$$
is a Bessel process of index  $\nu=-1/3$, i.e. of dimension
$\delta=2+2\nu=4/3$. 

Since $Y_t = \phi(\rho_t)=(\frac32 \rho_t)^{2/3}$, its excursion measure
$n^Y= n^\rho\circ \phi^{-1}$ is the image of $n^Y$ by $\phi$.
The excursion measure of $n^\rho$ normalized by its length is well
defined (see Corollary \ref{cor:excmeabessellength}) and equals $\pi^\rho=P^{\delta',1}_{0,0}$ with 
$\delta'=4-\delta=8/3$.

We claim that $\pi^Y=  \pi^\rho\circ \phi^{-1}=\gamma_\cE$. Observe that
the normalizing operator is 
$\lambda(w)= R(w)^{-1/3} w(t R(w))$ (this operator is adapted
to the scaling of $Y$). Indeed, for positive measurable
$F,\psi$, we have, by the disintegration property for $n^\rho$, since
$\lambda \circ \phi = \phi \circ \nu$ and $\xi\circ \phi = \xi$
\begin{align}\label{eqim2}
\nonumber  n^Y(F\circ \lambda \; \psi(R)) &= \nrho{F\circ\lambda\circ
    \phi \; \psi(R)}=\intof
  \pi^\rho_t(F\circ \phi\circ \nu) \, \psi(t) \,\nrho{R \in dt} \hspace{-4cm} & &\\ 
\nonumber &= \intof \pi^\rho(F\circ \phi\circ \nu \circ s_{1/t})\, \psi(t)
n^Y(R \in dt)\\
\nonumber &= \intof \pi^\rho(F\circ \phi\circ \nu)\, \psi(t)
n^Y(R \in dt)\quad \quad \text{(since $\nu \circ s_c = \nu$)}\\
\nonumber &= \intof \pi^\rho(F\circ \phi)\, \psi(t)
n^Y(R \in dt), \quad \quad \quad \text{(since under $\pi^\rho$ we have
  $\nu(w)=w$)}\\
&= \pi^\rho(F\circ \phi) \times n^Y(\psi(R))=  \pi^\rho \circ \phi^{-1}(F)\times  n^Y(\psi(R)).
\end{align}

We use now Lemma \ref{lem:excy} to identify this probability measure:
$$  n^Y(F\circ \lambda \; \psi(R)) = \nplus{F\circ \lambda(w\circ
  a_w)\; \psi(R(w\circ a_w))}.$$
Observe that for the excursion $u=w\circ a_w$ we have $R(u)=A(w)$,
$ A(u)=1$ and
$$ \lambda(u)(s) = R(u)^{-1/3} u(s R(u)) = A(w)^{-1/3} w\circ
a_w(s A(w)) = v \circ a_v(s),$$
with $v=\nu^1(w)$ because 
$$A_t(v)= \int_0^t A^{-1/3} w(s A^{2/3})\, ds = \unsur{A(w)} A_{t
  A(w)^{2/3}}(w),$$ and thus
$ a_v(s)= A(w)^{-2/3} a_w(s A(w)).$
Therefore, by the independence with respect to the area formula,
\begin{align}\label{eqim3}
\nonumber  n^Y(F\circ \lambda \; \psi(R)) &=\nplus{ F(v\circ a_v)
     \psi(A(w))} = \esp{F(\cE_{a_\cE})} \, \nplus{\psi(A)}=\gamma_\cE(F)\, \nplus{\psi(A)}.
\end{align}
Combining the preceding equality with  \eqref{eqim2}  we obtain
\begin{equation}\label{eqbm}
 \gamma_\cE(F)\,
\nplus{\psi(A)}= n^Y(F\circ \lambda \; \psi(R) )= \pi^\rho\circ \phi^{-1}(F) \times n^Y(\psi(R)).
\end{equation}
By the definition of the normalizing measure, \eqref{eqbm}
yields  $\pi^Y= \gamma_\cE = \pi^\rho\circ \phi^{-1}.$

\iflong

\appendix

\section{Existence of a smooth joint density for Brownian excursion
  area and position}\label{app:den}
\newcommand{\cunb}{C^1_b}
{\color{black} We denote by $\cunb$ the set of bounded and continuously differentiable functions,
whose partial derivatives are bounded.}
\begin{lemma}\label{lem:app:den}
  Let $(e_s,0\le s\le 1)$ be the  standard Brownian excursion. Let
  $A(e)=\int_0^1 e_s\,ds$ be its area and let $0< t_1 < t_2 < \ldots
  < t_d\le 1$. Then the random vector $(A(e); e_{t_i} \, 1\le i\le d)$ has
  a $\cunb$ joint density.
\end{lemma}

\begin{proof*}
The law of $e$ is 
$P^{3,1}_{0,0}$ the law of the Bessel bridge of dimension 3 of length
$1$. It is standard knowledge, se e.g. \cite{RY91}, that under
$P^{3,1}_{0,0}$ the coordinate vector 
$(X(t_1), \ldots, X(t_d))$ has a $\cunb$ density $\Lambda(x_1,
\ldots, x_n)$ (since $P^{3,1}_{0,0}$ is the law of the norm of a three
dimensional Brownian bridge). By Markov's property, for any positive
measurable functions $\psi_i$ and $\lambda>0$, we have

\begin{align*}
  P^{3,1}_{0,0}\etp{ e^{-\lambda A_1}\prod \psi_i(X(t_i)) }&= 
\int \Lambda(x_1,
\ldots, x_n) \prod \psi_i(x_i) \prod
P^{3,t_{i+1}-t_i}_{x_i,x_{i+1}}\etp{e^{-\lambda A_{t_{i+1}-t_i}}}.
\end{align*}
Therefore it remains to prove that for any
 $x,t$, under  $P^{3}_{x}$, the  couple $(X_t, A_t=\int_0^t X_s\,ds)$
 has a $\cunb$ density.
 

 Since $P^3_x$ is the  $h$-transform
for the harmonic function  $h(x)=x$ of the law  $Q_x$ of Brownian
motion killed when it reaches $0$,  (see \cite[Chapter VIII,
Section 3]{RY91}), on  $\cF_t$, $\frac{dP^3_x}{dQ_x} = X_t$ and for
every measurable positive $f,g$:

\begin{equation*}
  P^3_x\etc{f(X_t) g(A_t)} = \unsur{x} Q_x\etc{X_t f(X_t) g(A_t)}
\end{equation*}
It remains to show that $(X_t,A_t)$ has a joint density under  $Q_x$
and this is quite straightforward by a hitting time decomposition (see
e.g. \cite{Lachal91} or \cite{Goldman71}) which we reproduce here. 

Let $T_0=\inf\ens{t>0: B_t=0}$
be the hitting time of  $0$ by a brownian motion $B$ and 
$p_t(x,a;y,b)$ be the transition density of the Markov centered
gaussian process  $(B_t,
A_t(B))$. Then,  by Strong Markov's property:
\begin{align*}
   Q_x\etc{ f(X_t) g(A_t)}&=\esperance{x}{f(B_t) g(A_t(B))\un{t< T_0}}
   = \esperance{x}{f(B_t) g(A_t(B))} -
   \esperance{x}{f(B_t)g(A_t(B)) \un{t \ge T_0}}\\
&=\int p_t(x,0;u,v) f(u) g(v) \, du dv \\
&\quad - \int \PP_x\etp{T_0 \in ds,
  A_{T_0}(B) \in da} \un{s<t} p_{t-s}(0,a;u,v) f(u) g(v+a)\, du dv
\end{align*}
\end{proof*}
\fi

  \bibliographystyle{plain}
  \bibliography{cnp}

\affiliationone{
   P. Carmona and N. Pétrélis\\
   Université de Nantes, France
   \email{philippe.carmona@univ-nantes.fr\\
   nicolas.petrelis@univ-nantes.fr}}

\end{document}